\documentclass[11pt]{amsart}

\usepackage{amsmath}
\usepackage{amssymb}
\usepackage{amscd}
\usepackage{color}
\usepackage{hyperref}
\usepackage{mathrsfs}
\usepackage{bm}

\usepackage{xy}
\xyoption{all}

\newcommand{\nc}{\newcommand}

\nc{\CC}{{\mathbb{C}}}
\nc{\FF}{{\mathbb{F}}}
\nc{\HH}{{\mathbb{H}}}
\nc{\LL}{{\mathbb{L}}}
\nc{\RR}{{\mathbb{R}}}
\nc{\PP}{{\mathbb{P}}}
\nc{\OO}{{\mathbb{O}}}

\nc{\QQ}{{\mathbb{Q}}}
\nc{\ZZ}{{\mathbb{Z}}}

\nc{\cA}{{\mathscr{A}}}
\nc{\cB}{{\mathscr{B}}}
\nc{\cC}{{\mathscr{C}}}
\nc{\cD}{{\mathscr{D}}}
\nc{\cE}{{\mathscr{E}}}
\nc{\cF}{{\mathscr{F}}}
\nc{\cG}{{\mathscr{G}}}
\nc{\cH}{{\mathscr{H}}}
\nc{\cI}{{\mathscr{I}}}
\nc{\cJ}{{\mathscr{J}}}
\nc{\cK}{{\mathscr{K}}}
\nc{\cL}{{\mathscr{L}}}
\nc{\cM}{{\mathscr{M}}}
\nc{\cN}{{\mathscr{N}}}
\nc{\cO}{{\mathscr{O}}}
\nc{\cP}{{\mathscr{P}}}
\nc{\cQ}{{\mathscr{Q}}}
\nc{\cR}{{\mathscr{R}}}
\nc{\cS}{{\mathscr{S}}}
\nc{\cT}{{\mathscr{T}}}
\nc{\cU}{{\mathscr{U}}}
\nc{\cV}{{\mathscr{V}}}
\nc{\cW}{{\mathscr{W}}}
\nc{\cX}{{\mathscr{X}}}
\nc{\cY}{{\mathscr{Y}}}
\nc{\cZ}{{\mathscr{Z}}}

\nc{\bA}{{\mathbf{A}}}
\nc{\bB}{{\mathbf{B}}}
\nc{\bC}{{\mathbf{C}}}
\nc{\bD}{{\mathbf{D}}}
\nc{\bE}{{\mathbf{E}}}
\nc{\bF}{{\mathbf{F}}}
\nc{\bG}{{\mathbf{G}}}
\nc{\bH}{{\mathbf{H}}}
\nc{\bI}{{\mathbf{I}}}
\nc{\bJ}{{\mathbf{J}}}
\nc{\bK}{{\mathbf{K}}}
\nc{\bL}{{\mathbf{L}}}
\nc{\bM}{{\mathbf{M}}}
\nc{\bN}{{\mathbf{N}}}
\nc{\bO}{{\mathbf{O}}}
\nc{\bP}{{\mathbf{P}}}
\nc{\bQ}{{\mathbf{Q}}}
\nc{\bR}{{\mathbf{R}}}
\nc{\bS}{{\mathbf{S}}}
\nc{\bT}{{\mathbf{T}}}
\nc{\bU}{{\mathbf{U}}}
\nc{\bV}{{\mathbf{V}}}
\nc{\bW}{{\mathbf{W}}}
\nc{\bX}{{\mathbf{X}}}
\nc{\bY}{{\mathbf{Y}}}
\nc{\bZ}{{\mathbf{Z}}}

\nc{\ba}{{\mathbf{a}}}
\nc{\bb}{{\mathbf{b}}}
\nc{\bc}{{\mathbf{c}}}
\nc{\bd}{{\mathbf{d}}}
\nc{\be}{{\mathbf{e}}}
\nc{\bg}{{\mathbf{g}}}
\nc{\bh}{{\mathbf{h}}}
\nc{\bi}{{\mathbf{i}}}
\nc{\bj}{{\mathbf{j}}}
\nc{\bk}{{\mathbf{k}}}
\nc{\bn}{{\mathbf{n}}}
\nc{\bo}{{\mathbf{o}}}
\nc{\bp}{{\mathbf{p}}}
\nc{\bq}{{\mathbf{q}}}
\nc{\br}{{\mathbf{r}}}
\nc{\bs}{{\mathbf{s}}}
\nc{\bt}{{\mathbf{t}}}
\nc{\bu}{{\mathbf{u}}}
\nc{\bv}{{\mathbf{v}}}
\nc{\bx}{{\mathbf{x}}}
\nc{\by}{{\mathbf{y}}}
\nc{\bz}{{\mathbf{z}}}

\nc{\fA}{{\mathfrak{A}}}
\nc{\fB}{{\mathfrak{B}}}
\nc{\fC}{{\mathfrak{C}}}
\nc{\fD}{{\mathfrak{D}}}
\nc{\fE}{{\mathfrak{E}}}
\nc{\fF}{{\mathfrak{F}}}
\nc{\fG}{{\mathfrak{G}}}
\nc{\fH}{{\mathfrak{H}}}
\nc{\fI}{{\mathfrak{I}}}
\nc{\fJ}{{\mathfrak{J}}}
\nc{\fK}{{\mathfrak{K}}}
\nc{\fL}{{\mathfrak{L}}}
\nc{\fM}{{\mathfrak{M}}}
\nc{\fN}{{\mathfrak{N}}}
\nc{\fO}{{\mathfrak{O}}}
\nc{\fP}{{\mathfrak{P}}}
\nc{\fQ}{{\mathfrak{Q}}}
\nc{\fR}{{\mathfrak{R}}}
\nc{\fS}{{\mathfrak{S}}}
\nc{\fT}{{\mathfrak{T}}}
\nc{\fU}{{\mathfrak{U}}}
\nc{\fV}{{\mathfrak{V}}}
\nc{\fW}{{\mathfrak{W}}}
\nc{\fX}{{\mathfrak{X}}}
\nc{\fY}{{\mathfrak{Y}}}
\nc{\fZ}{{\mathfrak{Z}}}

\nc{\fa}{{\mathfrak{a}}}
\nc{\fb}{{\mathfrak{b}}}
\nc{\fc}{{\mathfrak{c}}}
\nc{\fd}{{\mathfrak{d}}}
\nc{\fe}{{\mathfrak{e}}}
\nc{\ff}{{\mathfrak{f}}}
\nc{\fg}{{\mathfrak{g}}}
\nc{\fh}{{\mathfrak{h}}}
\nc{\fj}{{\mathfrak{j}}}
\nc{\fk}{{\mathfrak{k}}}
\nc{\fl}{{\mathfrak{l}}}
\nc{\fm}{{\mathfrak{m}}}
\nc{\fn}{{\mathfrak{n}}}
\nc{\fo}{{\mathfrak{o}}}
\nc{\fp}{{\mathfrak{p}}}
\nc{\fq}{{\mathfrak{q}}}
\nc{\fr}{{\mathfrak{r}}}
\nc{\fs}{{\mathfrak{s}}}
\nc{\ft}{{\mathfrak{t}}}
\nc{\fu}{{\mathfrak{u}}}
\nc{\fv}{{\mathfrak{v}}}
\nc{\fw}{{\mathfrak{w}}}
\nc{\fx}{{\mathfrak{x}}}
\nc{\fy}{{\mathfrak{y}}}
\nc{\fz}{{\mathfrak{z}}}

\nc{\sA}{{\mathsf{A}}}
\nc{\sB}{{\mathsf{B}}}
\nc{\sC}{{\mathsf{C}}}
\nc{\sD}{{\mathsf{D}}}
\nc{\sE}{{\mathsf{E}}}
\nc{\sF}{{\mathsf{F}}}
\nc{\sG}{{\mathsf{G}}}
\nc{\sH}{{\mathsf{H}}}
\nc{\sI}{{\mathsf{I}}}
\nc{\sJ}{{\mathsf{J}}}
\nc{\sK}{{\mathsf{K}}}
\nc{\sL}{{\mathsf{L}}}
\nc{\sM}{{\mathsf{M}}}
\nc{\sN}{{\mathsf{N}}}
\nc{\sO}{{\mathsf{O}}}
\nc{\sP}{{\mathsf{P}}}
\nc{\sQ}{{\mathsf{Q}}}
\nc{\sR}{{\mathsf{R}}}
\nc{\sS}{{\mathsf{S}}}
\nc{\sT}{{\mathsf{T}}}
\nc{\sU}{{\mathsf{U}}}
\nc{\sV}{{\mathsf{V}}}
\nc{\sW}{{\mathsf{W}}}
\nc{\sX}{{\mathsf{X}}}
\nc{\sY}{{\mathsf{Y}}}
\nc{\sZ}{{\mathsf{Z}}}

\nc{\sa}{{\mathsf{a}}}
\nc{\sd}{{\mathsf{d}}}
\nc{\se}{{\mathsf{e}}}
\nc{\sg}{{\mathsf{g}}}
\nc{\sh}{{\mathsf{h}}}
\nc{\si}{{\mathsf{i}}}
\nc{\sj}{{\mathsf{j}}}
\nc{\sk}{{\mathsf{k}}}
\nc{\sm}{{\mathsf{m}}}
\nc{\sn}{{\mathsf{n}}}
\nc{\so}{{\mathsf{o}}}
\nc{\sq}{{\mathsf{q}}}
\nc{\sr}{{\mathsf{r}}}
\nc{\st}{{\mathsf{t}}}
\nc{\su}{{\mathsf{u}}}
\nc{\sv}{{\mathsf{v}}}
\nc{\sw}{{\mathsf{w}}}
\nc{\sx}{{\mathsf{x}}}
\nc{\sy}{{\mathsf{y}}}
\nc{\sz}{{\mathsf{z}}}

\nc{\oA}{{\overline{A}}}
\nc{\oB}{{\overline{B}}}
\nc{\oC}{{\overline{C}}}
\nc{\oD}{{\overline{D}}}
\nc{\oE}{{\overline{E}}}
\nc{\oF}{{\overline{F}}}
\nc{\oG}{{\overline{G}}}
\nc{\oH}{{\overline{H}}}
\nc{\oI}{{\overline{I}}}
\nc{\oJ}{{\overline{J}}}
\nc{\oK}{{\overline{K}}}
\nc{\oL}{{\overline{L}}}
\nc{\oM}{{\overline{M}}}
\nc{\oN}{{\overline{N}}}
\nc{\oO}{{\overline{O}}}
\nc{\oP}{{\overline{P}}}
\nc{\oQ}{{\overline{Q}}}
\nc{\oR}{{\overline{R}}}
\nc{\oS}{{\overline{S}}}
\nc{\oT}{{\overline{T}}}
\nc{\oU}{{\overline{U}}}
\nc{\oV}{{\overline{V}}}
\nc{\oW}{{\overline{W}}}
\nc{\oX}{{\overline{X}}}
\nc{\oY}{{\overline{Y}}}
\nc{\oZ}{{\overline{Z}}}

\nc{\oa}{{\overline{a}}}
\nc{\ob}{{\overline{b}}}
\nc{\oc}{{\overline{c}}}
\nc{\od}{{\overline{d}}}
\nc{\of}{{\overline{f}}}
\nc{\og}{{\overline{g}}}
\nc{\oh}{{\overline{h}}}
\nc{\oi}{{\overline{i}}}
\nc{\oj}{{\overline{j}}}
\nc{\ok}{{\overline{k}}}
\nc{\ol}{{\overline{l}}}
\nc{\om}{{\overline{m}}}
\nc{\on}{{\overline{n}}}
\nc{\oo}{{\overline{o}}}
\nc{\op}{{\overline{p}}}
\nc{\oq}{{\overline{q}}}
\nc{\os}{{\overline{s}}}
\nc{\ot}{{\overline{t}}}
\nc{\ou}{{\overline{u}}}
\nc{\ov}{{\overline{v}}}
\nc{\ow}{{\overline{w}}}
\nc{\ox}{{\overline{x}}}
\nc{\oy}{{\overline{y}}}
\nc{\oz}{{\overline{z}}}

\nc{\tA}{{\tilde{A}}}
\nc{\tB}{{\tilde{B}}}
\nc{\tC}{{\tilde{C}}}
\nc{\tD}{{\tilde{D}}}
\nc{\tE}{{\tilde{E}}}
\nc{\tF}{{\tilde{F}}}
\nc{\tG}{{\tilde{G}}}
\nc{\tH}{{\tilde{H}}}
\nc{\tI}{{\tilde{I}}}
\nc{\tJ}{{\tilde{J}}}
\nc{\tK}{{\tilde{K}}}
\nc{\tL}{{\tilde{L}}}
\nc{\tM}{{\tilde{M}}}
\nc{\tN}{{\tilde{N}}}
\nc{\tO}{{\tilde{O}}}
\nc{\tP}{{\tilde{P}}}
\nc{\tQ}{{\tilde{Q}}}
\nc{\tR}{{\tilde{R}}}
\nc{\tS}{{\tilde{S}}}
\nc{\tT}{{\tilde{T}}}
\nc{\tU}{{\tilde{U}}}
\nc{\tV}{{\tilde{V}}}
\nc{\tW}{{\tilde{W}}}
\nc{\tX}{{\tilde{X}}}
\nc{\tY}{{\tilde{Y}}}
\nc{\tZ}{{\tilde{Z}}}

\nc{\ta}{{\tilde{a}}}
\nc{\tb}{{\tilde{b}}}
\nc{\tc}{{\tilde{c}}}
\nc{\td}{{\tilde{d}}}
\nc{\te}{{\tilde{e}}}
\nc{\tf}{{\tilde{f}}}
\nc{\tg}{{\tilde{g}}}
\nc{\ti}{{\tilde{i}}}
\nc{\tj}{{\tilde{j}}}
\nc{\tk}{{\tilde{k}}}
\nc{\tl}{{\tilde{l}}}
\nc{\tm}{{\tilde{m}}}
\nc{\tn}{{\tilde{n}}}
\nc{\tp}{{\tilde{p}}}
\nc{\tq}{{\tilde{q}}}
\nc{\tr}{{\tilde{r}}}
\nc{\ts}{{\tilde{s}}}
\nc{\tu}{{\tilde{u}}}
\nc{\tv}{{\tilde{v}}}
\nc{\tw}{{\tilde{w}}}
\nc{\tx}{{\tilde{x}}}
\nc{\ty}{{\tilde{y}}}
\nc{\tz}{{\tilde{z}}}

\nc{\hA}{{\hat{A}}}
\nc{\hB}{{\hat{B}}}
\nc{\hC}{{\hat{C}}}
\nc{\hD}{{\hat{D}}}
\nc{\hE}{{\hat{E}}}
\nc{\hF}{{\hat{F}}}
\nc{\hG}{{\hat{G}}}
\nc{\hH}{{\hat{H}}}
\nc{\hI}{{\hat{I}}}
\nc{\hJ}{{\hat{J}}}
\nc{\hK}{{\hat{K}}}
\nc{\hL}{{\hat{L}}}
\nc{\hM}{{\hat{M}}}
\nc{\hN}{{\hat{N}}}
\nc{\hO}{{\hat{O}}}
\nc{\hP}{{\hat{P}}}
\nc{\hQ}{{\hat{Q}}}
\nc{\hR}{{\hat{R}}}
\nc{\hS}{{\hat{S}}}
\nc{\hT}{{\hat{T}}}
\nc{\hU}{{\hat{U}}}
\nc{\hV}{{\hat{V}}}
\nc{\hW}{{\hat{W}}}
\nc{\hX}{{\hat{X}}}
\nc{\hY}{{\hat{Y}}}
\nc{\hZ}{{\hat{Z}}}

\nc{\ha}{{\hat{a}}}
\nc{\hb}{{\hat{b}}}
\nc{\hc}{{\hat{c}}}
\nc{\hd}{{\hat{d}}}
\nc{\he}{{\hat{e}}}
\nc{\hf}{{\hat{f}}}
\nc{\hg}{{\hat{g}}}
\nc{\hh}{{\hat{h}}}
\nc{\hi}{{\hat{i}}}
\nc{\hj}{{\hat{j}}}
\nc{\hk}{{\hat{k}}}
\nc{\hl}{{\hat{l}}}
\nc{\hn}{{\hat{n}}}
\nc{\ho}{{\hat{o}}}
\nc{\hp}{{\hat{p}}}
\nc{\hq}{{\hat{q}}}
\nc{\hr}{{\hat{r}}}
\nc{\hs}{{\hat{s}}}
\nc{\hu}{{\hat{u}}}
\nc{\hv}{{\hat{v}}}
\nc{\hw}{{\hat{w}}}
\nc{\hx}{{\hat{x}}}
\nc{\hy}{{\hat{y}}}
\nc{\hz}{{\hat{z}}}

\nc{\rA}{{\mathrm{A}}}
\nc{\rB}{{\mathrm{B}}}
\nc{\rE}{{\mathrm{E}}}
\nc{\rI}{{\mathrm{I}}}
\nc{\rS}{{\mathrm{S}}}
\nc{\rW}{{\mathrm{W}}}

\nc{\bl}{\bm{\lambda}}
\nc{\eps}{\varepsilon}
\nc{\lan}{\big\langle}
\nc{\ran}{\big\rangle}
\nc{\kk}{{\mathsf{k}}}

\DeclareMathOperator{\rL}{\mathrm{L}}

\def\bw#1#2{\textstyle{\bigwedge\hskip-0.9mm^{#1}}\hskip0.2mm{#2}}

\DeclareMathOperator{\Hom}{\mathrm{Hom}}
\DeclareMathOperator{\Ext}{\mathrm{Ext}}

\DeclareMathOperator{\Ker}{\mathrm{Ker}}

\DeclareMathOperator{\Ima}{\mathrm{Im}}

\DeclareMathOperator{\NS}{\mathrm{NS}}
\DeclareMathOperator{\Sym}{\mathrm{Sym}}

\DeclareMathOperator{\Gro}{\mathrm{K_0}}
\DeclareMathOperator{\NGr}{\mathrm{K_0^{num}}}
\DeclareMathOperator{\KGr}{\mathrm{K_0}}

\theoremstyle{plain}

\newtheorem{theorem}{Theorem}[section]

\newtheorem{lemma}[theorem]{Lemma}
\newtheorem{proposition}[theorem]{Proposition}
\newtheorem{corollary}[theorem]{Corollary}

\theoremstyle{definition}

\newtheorem{definition}[theorem]{Definition}

\newtheorem{example}[theorem]{Example}

\theoremstyle{remark}

\newtheorem{remark}[theorem]{Remark}

\newcommand{\GG}{\mathrm{G}}
\nc{\ce}{{\mathrm{e}}}
\nc{\rc}{{\mathrm{c}}}
\nc{\rr}{{\mathrm{r}}}
\nc{\rs}{{\mathrm{s}}}

\DeclareMathOperator{\rk}{{\mathrm{rk}}}
\DeclareMathOperator{\CH}{{\mathrm{CH}}}
\DeclareMathOperator{\Br}{{\mathrm{Br}}}
\DeclareMathOperator{\Conv}{{\mathrm{Conv}}}

\title[Exceptional collections in surface-like categories]{Exceptional collections in surface-like categories}
\author{Alexander Kuznetsov}
\address{{\sloppy
\parbox{0.9\textwidth}{
Algebraic Geometry Section, Steklov Mathematical Institute of Russian Academy of Sciences,\\
8 Gubkin str., Moscow 119991 Russia
}\bigskip}}
\email{akuznet@mi.ras.ru}
\date{}
\thanks{This work is supported by the Russian Science Foundation under grant 14-50-00005.}

\begin{document}

\begin{abstract}
We provide a categorical framework for recent results of Markus Perling on combinatorics of exceptional collections on numerically rational surfaces.
Using it we simplify and generalize some of Perling's results as well as Vial's criterion for existence of a numerical exceptional collection.
\end{abstract}

\maketitle

\section{Introduction}

A beautiful recent paper of Markus Perling~\cite{Pe} proves that any numerically exceptional collection of maximal length in the derived category of a numerically rational surface 
(i.e., a surface with zero irregularity and geometric genus) can be transformed by mutations into an exceptional collection consisting of objects of rank 1.
In this note we provide a categorical framework 
that allows to extend and simplify Perling's result.

For this we introduce a notion of a surface-like category.
Roughly speaking, it is a triangulated category~$\cT$ whose \emph{numerical Grothendieck group} $\NGr(\cT)$, considered as an abelian group with a bilinear form (\emph{Euler form}), 
behaves similarly to the numerical Grothendieck group of a smooth projective surface, see Definition~\ref{def:surface-like}.
Of course, for any smooth projective surface $X$ its derived category $\bD(X)$ is surface-like. 
However, there are surface-like categories of different nature, for instance derived categories of noncommutative surfaces and of even-dimensional Calabi--Yau varieties turn out to be surface-like (Example~\ref{example:cy}).
Also, some subcategories of surface-like categories are surface-like.
Thus, the notion of a surface-like category is indeed more general.

In fact, all results of this paper have a numerical nature, so instead of considering categories, we pass directly to their numerical Grothendieck groups.
These are free abelian groups $\GG$ (we assume them to be of finite rank), equipped with a bilinear form $\chi$ that is neither symmetric, nor skew-symmetric in general.
We call such pair $(\GG,\chi)$ a \emph{pseudolattice} (since this is a non-symmetric version of a lattice).
We define and investigate a notion of a \emph{surface-like} pseudolattice (Definition~\ref{def:surface-like}), and show that it has many features similar to numerical Grothendieck groups of surfaces.
For instance, one can define the rank function on such~$\GG$, define the Neron--Severi lattice $\NS(\GG)$ of $\GG$ (that is isomorphic to the Neron--Severi lattice of the surface~$X$ when $\GG =\NGr(\bD(X))$), 
and construct the canonical class $K_\GG \in \NS(\GG) \otimes \QQ$ (in general, it is only rational).

We also introduce some important properties of surface-like pseudolattices: \emph{geometricity} (Definition~\ref{def:geometricity}), \emph{minimality} (Definition~\ref{def:minimlaity}),
and define their \emph{defects} (Definition~\ref{def:defect}).
The main result of this paper is Theorem~\ref{theorem-ranks-1}, saying that if a geometric surface-like pseudolattice with zero defect has an exceptional basis (Definition~\ref{def:exceptional}),
then this basis can be transformed by mutations to a basis consisting of elements of rank~1.

To prove Theorem~\ref{theorem-ranks-1} we first classify all minimal geometric pseudolattices $\GG$ with an exceptional basis --- it turns out that minimality implies that 
the pseudolattice is isometric either to~$\NGr(\bD(\PP^2))$ or to $\NGr(\bD(\PP^1 \times \PP^1))$, see Theorem~\ref{theorem:minimal-cats}.
In particular, its rank is 3 or 4, and its defect is 0.

To get the general case from this we investigate a kind of \emph{the minimal model program} for surface-like pseudolattices:
we introduce the notion of a \emph{contraction} of a pseudolattice (with respect to an exceptional element of zero rank),
and show that one can always pass from a general surface-like pseudolattice to a minimal one by a finite number of contractions.
We verify that geometricity is preserved under contractions, and that defect does not decrease.
In particular, if we start with a geometric pseudolattice of zero defect with an exceptional basis, then defect does not change under these contractions.
This allows to deduce Theorem~\ref{theorem-ranks-1} from Theorem~\ref{theorem:minimal-cats}.

In most of the proofs we follow the original arguments of Perling.
The main new feature that we introduce is the notion of a surface-like pseudolattice that allows to define the contraction operation and gives more flexibility.
In particular, this allows to get rid easily from exceptional objects of zero rank that appear as a headache in Perling's approach.
Also, the general categorical perspective we take simplifies some of the computations, especially those related to use of Riemann--Roch theorem.

Besides Perling's results, we also apply this technique to prove a generalization of a criterion of Charles Vial \cite[Theorem~3.1]{V} 
for existence of a numerically exceptional collection in the derived category of a surface, see Theorem~\ref{theorem:criterion}.
In fact, in the proof we use a lattice-theoretic result~\cite[Proposition~A.12]{V}, but besides that, the proof is an elementary consequence of the minimal model program for surface-like pseudolattices.

Of course, it would be very interesting to find higher-dimensional analogues of this technique.
For this, we need to understand well the relation between the (numerical) Grothendieck group of higher dimensional varieties and their (numerical) Chow groups.
An important result in this direction is proved in a recent paper by Sergey Gorchinskiy~\cite{G}.

The paper is organized as follows.
In Section~\ref{section:pseudolattices} we discuss numerical Grothendieck groups of triangulated categories and define pseudolattices.
We also discuss here exceptional bases of pseudolattices and their mutations.
In Section~\ref{section:surface-like} we define surface-like categories and pseudolattices, provide some examples, and discuss their basic properties.
In particular, we explain how one defines the rank function, the Neron--Severi lattice, and the canonical class of a surface-like pseudolattice.
In Section~\ref{section:minimal} we define minimality and geometricity, and classify minimal geometric surface-like pseudolattices with an exceptional basis.
In course of classification we associate with an exceptional basis a toric system and construct from it a fan in a rank 2 lattice, giving rise to a toric surface.
Finally, in section~\ref{section:mmp} we define a contraction of a pseudolattice with respect to a zero rank exceptional vector, and via a minimal model program 
deduce the main results of the paper from the classification results of the previous section.
Besides that, we define the defect of a pseudolattice and investigate its behavior under contractions.

\medskip

After the first version of this paper was published, I was informed about a paper of Louis de Thanhoffer de Volcsey and Michel Van den Bergh~\cite{dTVdB}, 
where a very similar categorical framework was introduced.
In particular, in~\cite{dTVdB} a notion of a \emph{Serre lattice of surface type} was defined, which is almost equivalent to the notion of a surface-like pseudolattice, see Remark~\ref{remark:surface-type-like},
and some numerical notions of algebraic geometry were developed on this base.
So, the content of Section~\ref{section:surface-like} of this paper is very close to the content of~\cite[Section~3]{dTVdB}.

\medskip

{\bf Acknowledgements.}
It should be clear from the above that the paper owns its existence to the work of Markus Perling.
I would also like to thank Sergey Gorchinskiy for very useful discussions.
I am very grateful to Pieter Belmans and Michel Van den Bergh for informing me about the paper~\cite{dTVdB}.

\section{Numerical Grothendieck groups and pseudolattices}\label{section:pseudolattices}

Let $\cT$ be a saturated (i.e., smooth and proper) $\kk$-linear triangulated category, where $\kk$ is a field.
Let $\KGr(\cT)$ be the Grothendieck group of $\cT$ and $\chi:\KGr(\cT) \otimes \KGr(\cT) \to \ZZ$ the Euler bilinear form:
\begin{equation*}
\chi(F_1,F_2) = \sum (-1)^i \dim \Hom(F_1,F_2[i]).
\end{equation*}
In general the form $\chi$ is neither symmetric, nor skew-symmetric; 
however, it is symmetrized by the Serre functor $\bS \colon \cT \to \cT$ of $\cT$, i.e., we have
\begin{equation*}
\chi(F_1,F_2) = \chi(F_2,\bS(F_1)).
\end{equation*}
Since the Serre functor is an autoequivalence, it follows that the left kernel of $\chi$ coincides with its right kernel.
We denote the quotient by 
\begin{equation*}
\NGr(\cT) := \KGr(\cT)/\Ker\chi;
\end{equation*}
it is called the {\sf numerical Grothendieck group} of $\cT$.

The numerical Grothendieck group $\NGr(\cT)$ is torsion-free (any torsion element would be in the kernel of $\chi$) and finitely generated \cite{E}, hence a free abelian group of finite rank.
The form $\chi$ induces a nondegenerate bilinear form on~$\NGr(\cT)$ which we also denote by $\chi$.
This form $\chi$ is still neither symmetric, nor skew-symmetric.

\subsection{Pseudolattices}

For purposes of this paper we want to axiomatize the above situation.

\begin{definition}
A {\sf pseudolattice} is a finitely generated free abelian group $\GG$ equipped with a nondegenerate bilinear form $\chi \colon \GG \otimes \GG \to \ZZ$.
A pseudolattice $(\GG,\chi)$ is {\sf unimodular} if the form $\chi$ induces an isomorphism $\GG \to \GG^\vee$.
An {\sf isometry} of pseudolattices $(\GG,\chi)$ and $(\GG',\chi')$ is an isomorphism of abelian groups $f \colon \GG \to \GG'$ such that $\chi'(f(v_1), f(v_2)) = \chi(v_1,v_2)$ for all $v_1,v_2 \in \GG$.
\end{definition}

For any $v_0 \in \GG$ we define 
\begin{equation*}
v_0^\perp = \{ v \in \GG \mid \chi (v_0,v) = 0 \},
\qquad
{}^\perp v_0 = \{ v \in \GG \mid \chi (v,v_0) = 0 \},
\end{equation*}
the {\sf right} and the {\sf left orthogonal} complements of $v_0$ in $\GG$.

We say that a pseudolattice $(\GG,\chi)$ {\sf has a Serre operator} if there is an automorphism $\rS_\GG \colon \GG \to \GG$ such that 
\begin{equation*}
\chi(v_1,v_2) = \chi(v_2,\rS_\GG(v_1))
\qquad\text{for all $v_1,v_2 \in \GG$}.
\end{equation*}
Of course, a Serre operator is unique, and if $\GG$ is unimodular then $\rS_\GG = (\chi^{-1})^T \circ \chi$ is the Serre operator.
Also it is clear that if $\GG = \NGr(\cT)$ and $\cT$ admits a Serre functor then the induced operator on $\GG$ is the Serre operator.
The notion of a pseudolattice with a Serre operator is equivalent to the notion of Serre lattice of~\cite{dTVdB}.

We denote by 
\begin{equation*}
\chi_+ := \chi + \chi^T,
\qquad 
\chi_- := \chi - \chi^T
\end{equation*}
the symmetrization and skew-symmetrization of the form $\chi$ on $\GG$.
If $\cT$ has a Serre operator, these forms can be written as
\begin{equation}\label{eq:serre-symmetrize}
\chi_+(v_1,v_2) = \chi(v_1, (1 + \rS_\GG)v_2),
\qquad 
\chi_-(v_1,v_2) = \chi(v_1, (1 - \rS_\GG)v_2).
\end{equation}

\subsection{Exceptional collections and mutations}

In this section we discuss a pseudolattice version of exceptional collections and mutations.

\begin{definition}\label{def:exceptional}
An element $\ce \in \GG$ is {\sf exceptional} if $\chi(\ce,\ce) = 1$.
A sequence of elements $(\ce_1,\ce_2,\dots,\ce_n)$ is {\sf exceptional} if each $\ce_i$ is exceptional and $\chi(\ce_i,\ce_j) = 0$ for all $i > j$ ({\sf semiorthogonality}).
An {\sf exceptional basis} is an exceptional sequence in $\GG$ that is its basis.
\end{definition}

Of course, when $\GG = \NGr(\cT)$, the class of an exceptional object in $\cT$ is exceptional, and the classes of elements of an exceptional collection in $\cT$ form an exceptional sequence in $\GG$,
which is an exceptional basis if the collection is full.

\begin{lemma}\label{lemma:chi-unimodular}
If $\GG$ has an exceptional sequence $\ce_1,\dots,\ce_n$ of length $n = \rk \GG$ then $\ce_1,\dots,\ce_n$ is an exceptional basis and $\GG$ is unimodular.
\end{lemma}
\begin{proof}
Consider the composition of maps 
\begin{equation*}
\ZZ^n \xrightarrow{\ (\ce_1,\dots,\ce_n)\ } \GG \xrightarrow{\quad\chi\quad} \GG^\vee \xrightarrow{\ (\ce_1,\dots,\ce_n)^T\ } \ZZ^n.
\end{equation*}
The composition is given by the Gram matrix of the form $\chi$ on the set of vectors $\ce_1,\dots,\ce_n$ which by definition of an exceptional sequence is upper-triangular with units on the diagonal, hence is an isomorphism.
It follows that the first map is injective and the last is surjective.
Since $\rk(\GG^\vee) = \rk(\GG) = n$ and $\GG^\vee$ is torsion-free, the last map is an isomorphism, hence so is the first map.
Thus $\ce_1,\dots,\ce_n$ is an exceptional basis in $\GG$.
Moreover, it follows that $\chi$ is an isomorphism, hence $\GG$ is unimodular.
\end{proof}

Assume $\ce \in \GG$ is exceptional.

\begin{definition}\label{def:mutation}
The {\sf left} (resp.\ {\sf right}) {\sf mutation} with respect to $\ce$ is an endomorphism of $\GG$ defined by 
\begin{equation}\label{eq:mutations}
\LL_\ce(v) := v - \chi(\ce,v)\ce,
\qquad 
\RR_\ce(v) := v - \chi(v,\ce)\ce.
\end{equation}
\end{definition}

In fact, $\LL_\ce$ is just the projection onto the right orthogonal $\ce^\perp$ (and similarly for $\RR_\ce$).
In particular, the mutations kill $\ce$ and define mutually inverse isomorphisms of the orthogonals
\begin{equation*}
\xymatrix@1@C=5em{
{}^\perp\ce\  \ar@<.5ex>[r]^{\LL_\ce} & \ \ce^\perp \ar@<.5ex>[l]^{\RR_\ce}.
}
\end{equation*}
Moreover, it is easy to see that given an exceptional sequence $\ce_\bullet = (\ce_1,\dots,\ce_n)$ in $\GG$, the sequences
\begin{equation*}
\begin{aligned}
\LL_{i,i+1}(\ce_\bullet) &:= (\ce_1, \dots, \ce_{i-1}, && \LL_{\ce_i}(\ce_{i+1}),\ \ \, \ce_i, 	&& \ce_{i+2}, \dots, \ce_n),\\
\RR_{i,i+1}(\ce_\bullet) &:= (\ce_1, \dots, \ce_{i-1}, && \ce_{i+1}, \RR_{\ce_{i+1}}(\ce_{i}),  && \ce_{i+2}, \dots, \ce_n)
\end{aligned}
\end{equation*}
are exceptional, and these two operations are mutually inverse.

If $\GG$ has a Serre operator, every exceptional sequence $\ce_1,\dots,\ce_n$ can be extended to an infinite sequence $\{\ce_i\}_{i \in \ZZ}$ by the rule
\begin{equation*}
\ce_i = \rS_\GG(\ce_{i+n}).
\end{equation*}
This sequence is called a {\sf helix}. Its main property is that for any $k \in \ZZ$ the sequence $\ce_{k+1},\dots,\ce_{k+n}$ is an exceptional sequence, generating the same helix (up to an index shift).

If $\ce_1,\dots,\ce_n$ is an exceptional basis then
\begin{equation*}
\ce_0 = (\LL_{\ce_1} \circ \dots \circ \LL_{\ce_{n-1}})(\ce_n),
\qquad 
\ce_{n+1} = (\RR_{\ce_n} \circ \dots \circ \RR_{\ce_{2}})(\ce_1),
\end{equation*}
and it follows that every element of the helix can be obtained from the original basis by mutations.
Mutations $\LL_{i,i+1}$ and $\RR_{i,i+1}$ (with $i$ now being an arbitrary integer) can be also defined for helices.

\section{Surface-like categories and pseudolattices}\label{section:surface-like}

\subsection{Definition and examples}

The next is the main definition of the paper.

\begin{definition}[\protect{cf.~\cite[Definition~3.2.1]{dTVdB}}]
\label{def:surface-like}
We say that a pseudolattice $(\GG,\chi)$ is {\sf surface-like} if there is a primitive element $\bp \in \GG$ such that
\begin{enumerate}
\item $\chi(\bp,\bp) = 0$,
\item $\chi_-(\bp,-) = 0$ (i.e., $\chi(\bp,v) = \chi(v,\bp)$ for any $v \in \GG$),
\item the form $\chi_-$ vanishes on the subgroup $\bp^\perp = {}^\perp\bp \subset \GG$ (i.e., $\chi$ is symmetric on $\bp^\perp$).
\end{enumerate}
An element $\bp$ as above is called a {\sf point-like element} in $\GG$.

We say that a smooth and proper triangulated category $\cT$ is {\sf surface-like} if its numerical Grothendieck group $(\NGr(\cT),\chi)$ is surface-like (with some choice of a point-like element).
\end{definition}

\begin{remark}\label{remark:surface-type-like}
A Serre operator $\rS_\GG$ of a surface-like pseudolattice, if exists, is unipotent by Corollary~\ref{corollary:serre-unipotent}, 
and by~\eqref{eq:serre-symmetrize} and nondegeneracy of~$\chi$ the rank of $\rS_\GG - 1$ does not exceed 2, 
hence a surface-like pseudolattice with a Serre operator is a Serre lattice of surface type as defined in~\cite[Definition~3.2.1]{dTVdB}.
Conversely, as a combination of Lemma~\ref{lemma:surface-like-criterion} below and~\cite[Lemma~3.3.2]{dTVdB} shows, a Serre lattice of surface* type is a surface-like pseudolattice (with a point-like element 
being a primitive generator of the smallest piece of the numerical codimension filtration of~\cite[Section~3.3]{dTVdB}).
However, again by Lemma~\ref{lemma:surface-like-criterion}, a Serre lattice of surface type which is not of surface* type is surface-like only if $\chi$ has an isotropic vector.
\end{remark}

If we choose a basis $v_0,\dots,v_{n-1}$ in $\GG$ such that $v_{n-1} = \bp$ and $\bp^\perp = \langle v_1,\dots,v_{n-1} \rangle$ 
then the above definition is equivalent to the fact, that the Gram matrix of the bilinear form $\chi$ takes the following form:
\begin{equation}\label{eq:chi-matrix}
\chi(v_i,v_j) = 
\begin{pmatrix}
a & b_1 & \dots & b_{n-2} & d \\
b'_1 & c_{11} & \dots & c_{1,n-2} & 0 \\
\vdots & \vdots & \ddots & \vdots & \vdots \\
b'_{n-2} & c_{n-2,1} & \dots & c_{n-2,,n-2} & 0 \\
d & 0 & \dots & 0 & 0
\end{pmatrix}
\end{equation}
with the submatrix $(c_{ij})$ being symmetric.

The following is a useful reformulation of Definition~\ref{def:surface-like}.

\begin{lemma}\label{lemma:surface-like-criterion}
A pseudolattice $\GG$ is surface-like if and only if one of the following two cases takes place:
\begin{enumerate}\renewcommand{\theenumi}{\alph{enumi}}
\item\label{item:cy} either $\chi_- = 0$ (i.e., the form $\chi$ is symmetric), and $\chi$ has an isotropic vector;
\item\label{item:r2} or the rank of $\chi_-$ equals $2$, and the restriction $\chi\vert_{\Ker \chi_-}$ is degenerate.
\end{enumerate}
In case~\eqref{item:cy} an element $\bp \in \GG$ is point-like if and only if it is isotropic, i.e., $\chi(\bp,\bp) = 0$.
In case~\eqref{item:r2} an element $\bp \in \GG$ is point-like if and only if $\bp \in \Ker(\chi\vert_{\Ker \chi_-})$.
\end{lemma}
\begin{proof}
Assume $\GG$ is surface-like.
If $\chi_- = 0$, we are in case~\eqref{item:cy}; then $\bp$ is isotropic by Definition~\ref{def:surface-like}(1).
Otherwise the rank of $\chi_-$ equals~2, since $\chi_-$ vanishes on a hyperplane $\bp^\perp \subset \GG$ by Definition~\ref{def:surface-like}(3), and moreover $\Ker\chi_- \subset \bp^\perp$.
Furthermore, $\bp \in \Ker \chi_-$ by Definition~\ref{def:surface-like}(2), and since $\Ker \chi_- \subset \bp^\perp$ we have $\chi(\Ker \chi_-, \bp) = 0$, hence $\bp$ is in the kernel of the restriction $\chi\vert_{\Ker\chi_-}$.

Conversely, if~\eqref{item:cy} holds and $\bp$ is isotropic, then $\chi_- = 0$ and clearly Definition~\ref{def:surface-like} holds.
Similarly, if~\eqref{item:r2} holds and $\bp \in \Ker (\chi\vert_{\Ker \chi_-})$, then parts (1) and (2) of Definition~\ref{def:surface-like} hold, 
and since $\bp^\perp$ contains $\Ker\chi_-$ as a hyperplane, part (3) also holds.
\end{proof}

The above argument also shows that when $\chi_- \ne 0$, part (1) of Definition~\ref{def:surface-like} follows from parts (2) and (3).
Note also that the restriction on the rank of $\chi_-$ also recently appeared in~\cite{BR1}.

Let us give some examples.

\begin{example}\label{example:cy}
Let $\cT$ be a smooth and proper Calabi--Yau category of even dimension (i.e., its Serre functor is a shift $\bS_\cT \cong [k]$ with even $k$), and $P \in \cT$ is a point object, i.e., $\Ext^\bullet(P,P)$ is an exterior algebra on $\Ext^1(P,P)$.
Then $\cT$ is a surface-like category and $\GG = \NGr(\cT)$ is a surface-like pseudolattice with $[P]$~being a point-like element, since the Euler form is symmetric and so we are in case~\eqref{item:cy} of Lemma~\ref{lemma:surface-like-criterion}.
\end{example}

For us, however, the main example is the next one.

\begin{example}\label{example:standard}
Let $X$ be a smooth projective surface, set $\cT = \bD(X)$ to be the bounded derived category of coherent sheaves on $X$, and let $\GG = \NGr(\cT) = \NGr(X)$.
The topological filtration $0 \subset F^2\Gro(X) \subset F^1\Gro(X) \subset F^0\Gro(X) = \Gro(X)$ of the Grothendieck group $\Gro(X)$ (by codimension of support)
induces a filtration $0 \subset \GG_2 \subset \GG_1 \subset \GG_0 = \GG$ on the numerical Grothendieck group. Consider the maps
\begin{equation*}\arraycolsep=.1em
\begin{array}{rll}
\rr & \colon \GG_0 / \GG_1 & \xrightarrow{\ \ } \ZZ, \\
\rc_1 & \colon \GG_1 / \GG_2 & \xrightarrow{\ \ \ } \NS(X), \\
\rs & \colon \GG_2 & \xrightarrow{\ \ \ } \ZZ,
\end{array}
\end{equation*}
given by the rank, the first Chern class, and the Euler characteristic, where $\NS(X)$ is the numerical Neron--Severi group of $X$ (in particular, we quotient out the torsion in $\CH^1(X)$).
The rank map is clearly an isomorphism, and so is $\rc_1$ (the surjectivity of $\rc_1$ follows from the surjectivity of $F^1\!K_0(X) \to \CH^1(X)$, and for injectivity it is enough
to note that if $\cL$ is a line bundle such that $\rc_1(\cL)$ is numerically equivalent to zero, then by Riemann--Roch $[\cL] = [\cO_X]$ in $\NGr(\bD(X))$).
It is also clear that $\rs$ is injective, so if we normalize it by dividing by the minimal degree of a $0$-cycle on $X$, the obtained map is an isomorphism.
Considering~$\rc_1$ as a linear map $\GG \to \NS(X)$ in a standard way, and extending linearly the normalized Euler characteristic map to a map~$\tilde{\rs} \colon \GG \to \ZZ$
(if $X$ has a 0-cycle of degree 1, we can take $\tilde\rs$ to be the Euler characteristic map),
we obtain an isomorphism
\begin{equation*}
\GG \xrightarrow{\ \sim\ } \ZZ \oplus \NS(X) \oplus \ZZ,
\qquad 
v \mapsto (\rr(v), \rc_1(v), \tilde{\rs}(v)).
\end{equation*}
A simple Riemann--Roch computation shows that the form $\chi_-$ is given by
\begin{equation*}
\chi_-((r_1,D_1,s_1),(r_2,D_2,s_2)) = r_1 (K_X\cdot D_2) - r_2(K_X \cdot D_1),
\end{equation*}
where $K_X \in \NS(X)$ is the canonical class of $X$ and $\cdot$ stands for the intersection pairing on $\NS(X)$.
The kernel of $\chi_-$ is spanned by all $(0,D,s) \in  \ZZ \oplus \NS(X) \oplus \ZZ$ such that $K_X \cdot D = 0$, in particular, the rank of $\chi_-$ is 2.
Furthermore,
by Riemann--Roch we have
\begin{equation}\label{eq:chi-surface}
\chi((0,D_1,s_1),(0,D_2,s_2)) = - D_1 \cdot D_2.
\end{equation}
Therefore, $\bp_X := (0,0,1)$ is contained in the kernel of $\chi\vert_{\Ker\chi_-}$, and since the intersection pairing on the numerical Neron--Severi group $\NS(X)$ is nondegenerate, $\bp_X$ generates this kernel unless $K_X^2 = 0$. 
In the latter case, the kernel of $\chi\vert_{\Ker\chi_-}$ is generated by $\bp_X$ and $K_X$.
In particular, Lemma~\ref{lemma:surface-like-criterion}\eqref{item:r2} shows that $\GG$ is surface-like with $\bp_X$ a point-like class, and that $\bp_X$ is the unique point-like class unless $K_X^2 = 0$.
\end{example}

We say that a surface-like category $\cT$ (resp.\ pseudolattice $\GG$) is {\sf standard}, if $(\cT,\bp) = (\bD(X),\bp_X)$ for a surface $X$ (resp.\ $(\GG,\bp) = (\NGr(\bD(X)),\bp_X)$).
Next example shows that even a standard surface-like pseudolattice of Example~\ref{example:standard} may have sometimes a non-standard surface-like structure.

\begin{example}
Assume again $X$ is a smooth projective surface, $\cT = \bD(X)$, and $\GG = \NGr(\cT)$.
In the notation of Example~\ref{example:standard} set $\bp_K := (0,K_X,0)$.
If $K_X^2 = 0$ then $\bp_K$ is a point-like element, that gives a different surface-like structure on the pseudolattice $\GG$ and the category $\bD(X)$.
\end{example}

Further we will need an explicit form of a standard pseudolattice for surfaces $X = \PP^2$ and $X = \PP^1 \times \PP^1$, 
and $X = \FF_1$ (the Hirzebruch ruled surface).

\begin{example}\label{example:p2-p1p1}
If $X = \PP^2$ then the classes of the sheaves $(\cO_{\PP^2},\cO_{\PP^2}(1),\cO_{\PP^2}(2))$ form an exceptional basis in which the Gram matrix of the Euler form looks like
\begin{equation*}
\chi_{\PP^2} = 
\left(\begin{smallmatrix}
1 & 3 & 6 \\ 0 & 1 & 3 \\ 0 & 0 & 1
\end{smallmatrix}\right).
\end{equation*}
If $X = \PP^1 \times \PP^1$ for each $c \in \ZZ$ the classes of the sheaves $(\cO_{\PP^1 \times \PP^1},\cO_{\PP^1 \times \PP^1}(1,0),\cO_{\PP^1 \times \PP^1}(c,1),\cO_{\PP^1 \times \PP^1}(c+1,1))$ 
form an exceptional basis in which the Gram matrix of the Euler form looks like
\begin{equation*}
\chi_{\PP^1 \times \PP^1} = \left(
\begin{smallmatrix}
1 & 2 & 2c + 2 & 2c + 4 \\ 0 & 1 & 2c & 2c + 2 \\ 0 & 0 & 1 & 2 \\ 0 & 0 & 0 & 1
\end{smallmatrix}
\right).
\end{equation*}
If $X = \FF_1$ for each $c \in \ZZ$ the classes of the sheaves $(\cO_{\FF^1},\cO_{\FF_1}(f),\cO_{\FF_1}(s+cf),\cO_{\FF_1}(s+(c+1)f))$ (where $f$ is the class of a fiber and $s$ is the class of the $(-1)$-section)
form an exceptional basis in which the Gram matrix of the Euler form looks like
\begin{equation*}
\chi_{\FF_1} = \left(
\begin{smallmatrix}
1 & 2 & 2c + 1 & 2c + 3 \\ 0 & 1 & 2c-1 & 2c + 1 \\ 0 & 0 & 1 & 2 \\ 0 & 0 & 0 & 1
\end{smallmatrix}
\right).
\end{equation*}
Note that the left mutation of $\cO_{\FF_1}(s+(c+1)f)$ through $\cO_{\FF_1}(s+cf)$ is isomorphic to $\cO_{\FF_1}(s+(c-1)f)$ (up to a shift, hence all such exceptional collections are mutation-equivalent.
Note also, that in case $c = 1$ the left mutation of $\cO_{\FF_1}(s+f)$ through $\cO_{\FF_1}(f)$ is isomorphic to the structure sheaf of the $(-1)$-section, in particular its rank is zero.
\end{example}

The next examples come from non-commutative geometry.

\begin{example}
Let $\cT$ be the derived category of a non-commutative projective plane (see, for instance, \cite{KKO}).
Then the numerical Grothendieck group $\NGr(\cT)$ is isometric to $\NGr(\PP^2)$ (the numerical Grothendieck group of a commutative plane), hence $\cT$ is a surface-like category.
The same applies to noncommutative $\PP^1 \times \PP^1$ and other noncommutative deformations of surfaces.
\end{example}

From this perspective it would be interesting to classify directed quivers whose derived categories of representations are surface-like
(for quivers with at most four vertices this is done in~\cite{dTVdB}).
Note that each of these categories can be realized as a semiorthogonal component of the derived category of a smooth projective variety \cite{Or2}, see also~\cite{BR2}.
More generally, one can ask when a gluing \cite{KL,Or1,Or3} of two triangulated categories is surface-like.

\begin{example}\label{example:brauer}
Let $X$ be a smooth projective complex K3 surface and $\beta \in \Br(X)$ an element in the Brauer group. 
Let $\cT = \bD(X,\beta)$ be the twisted derived category.
Its numerical Grothendieck group $\NGr(X,\beta)$ can be described as follows (see \cite[Section~1]{HS} for details).
Analyzing the cohomology in the exponential sequence, one can write the Brauer group as an extension
\begin{equation*}
0 \to \frac{H^2(X,\QQ)}{\NS(X)_\QQ + H^2(X,\ZZ)} \to \Br(X) \to H^3(X,\ZZ)_{\mathrm{torsion}} \to 0.
\end{equation*}
For a K3 surface the right term vanishes, hence a Brauer class $\beta$ can be represented by a rational cohomology class $B \in H^2(X,\QQ)$
(called a B-field).
Then 
\begin{equation*}
\NGr(X,\beta) = \{ (r,D,s) \in \QQ \oplus \NS(X)_\QQ \oplus \QQ \mid r \in \ZZ,\ D + r B \in \NS(X),\ s + D \cdot B + r B^2/2 \in \ZZ \},
\end{equation*}
and the Euler form is given by the Mukai pairing $\chi((r_1,D_1,s_1),(r_2,D_2,s_2)) = r_1s_2 - D_1\cdot D_2 + s_1r_2$. 
In this case $\bp_X = (0,0,1)$ is still a point-like element.
\end{example}

See other examples of surface-like categories in~\cite{BP,dTP}.

\subsection{Rank function and Neron--Severi lattice}

Assume $(\GG,\bp)$ is a surface-like pseudolattice and let~$\bp \in \GG$ be its point-like element.
The linear functions $\chi(\bp,-)$ and $\chi(-,\bp)$ on $\GG$ coincide (by Definition~\ref{def:surface-like}(2)). 
We define the {\sf rank function} associated with the point-like element $\bp$ by
\begin{equation*}
\br(-) := \chi(\bp,-) = \chi(-,\bp).
\end{equation*}
Note that $\bp^\perp = {}^\perp\bp = \Ker\br$.

\begin{lemma}\label{lemma:ns}
If $\GG$ is a surface-like pseudolattice and $\bp$ is its point-like element, there is a complex
\begin{equation}\label{eq:rp-p}
\ZZ \xrightarrow{\ \bp\ } \GG \xrightarrow{\ \br\ } \ZZ
\end{equation}
with injective $\bp$ and, if $\GG$ is unimodular, surjective $\br$.
Its middle cohomology 
\begin{equation}\label{eq:def-ns}
\NS(\GG) := \bp^\perp/\bp
\end{equation}
is a finitely generated free abelian group of rank $\rk(\GG) - 2$.
\end{lemma}
\begin{proof}
We have $\br(\bp) = \chi(\bp,\bp) = 0$ by Definition~\ref{def:surface-like}(1), hence~\eqref{eq:rp-p} is a complex.
The first map in~\eqref{eq:rp-p} is injective since $\GG$ is torsion-free.
The second map is nonzero, since $\chi$ is nondegenerate on~$\GG$.
If its image is $d\ZZ \subset \ZZ$ with $d \ge 2$ (this, up to a sign, is the same $d$ as in~\eqref{eq:chi-matrix}), then $\frac1d \br$ is a well defined element of~$\GG^\vee$.
If $\GG$ is unimodular, then there is an element $v \in \GG$ such that $d \cdot \chi(v,-) = \chi(\bp,-)$.
Therefore $\bp - d \cdot v$ is in the kernel of $\chi$, hence is zero, hence $\bp$ is not primitive.
This contradiction shows that $\br$ is surjective for unimodular $\GG$.

The group $\bp^\perp$ is torsion-free of rank $\rk(\GG) - 1$ since $\br \ne 0$, and the group $\NS(\GG)$ is torsion-free of rank $\rk(\GG) - 2$ since $\bp$ is primitive.
\end{proof}

The unimodularity assumption is not necessary for the surjectivity of $\br$, however $\br$ is not surjective in general.
Indeed, if $(\GG,\bp_X)$ is the surface-like pseudolattice of Example~\ref{example:brauer}, the index of $\br(\GG)$ in $\ZZ$ equals the order of $\beta$ in the Brauer group $\Br(X)$. 

The group $\NS(\GG)$ defined in~\eqref{eq:def-ns} is called {\sf Neron--Severi} group of a surface-like pseudolattice $\GG$.

\begin{lemma}\label{lemma:chi-ker-r}
If $\GG$ is a surface-like pseudolattice and $\bp$ is its point-like element, then the form $\chi$ induces a nondegenerate symmetric bilinear form on $\NS(\GG)$.
In other words, there is a unique nondegenerate symmetric form $q \colon \Sym^2 (\NS(\GG)) \to \ZZ$ such that the diagram
\begin{equation*}
\xymatrix{
\bp^\perp \otimes \bp^\perp
\ar@{->>}[rr] \ar[dr]_\chi &&
\NS(\GG) \otimes \NS(\GG) \ar[dl]^{-q} \\
& \ZZ
}
\end{equation*}
is commutative.
The lattice $\GG$ is unimodular if and only if $\br$ is surjective and $q$ is unimodular.
\end{lemma}

Note the sign convention chosen in order to make $q$ equal to the intersection pairing in the standard example (see~\eqref{eq:chi-surface}).
In terms of~\eqref{eq:chi-matrix}, the middle symmetric submatrix $(c_{ij})$ of $\chi$ is the Gram matrix of $-q$.

\begin{proof}
Recall that $\chi_-$ vanishes on $\bp^\perp$ by Definition~\ref{def:surface-like}(3), hence $\chi$ is symmetric on $\bp^\perp$.
Furthermore, if $v \in \bp^\perp$ then $\chi(\bp,v) = \br(v) = 0$, hence $\bp$ is in the kernel of $\chi$ on $\bp^\perp$, and thus $\chi$ induces a symmetric form $q$ on $\NS(\GG)$ such that the above diagram commutes.

Finally, consider the diagram
\begin{equation*}
\xymatrix{
\ZZ \ar@{=}[d] \ar[r]^-\bp & \GG \ar[r]^-{\br} \ar[d]^\chi & \ZZ \ar@{=}[d] \\
\ZZ \ar[r]^-{\br} & \GG^\vee \ar[r]^-{\bp} & \ZZ
}
\end{equation*}
which is commutative by definition of $\br$, hence is a morphism of complexes.
The map induced on the middle cohomologies of complexes $\NS(\GG) \to \NS(\GG)^\vee$ coincides with the map $-q$ by the definition of the latter.
Nondegeneracy of $\chi$ implies injectivity of that map, hence $q$ is nondegenerate.

If~$\GG$ is unimodular then the middle vertical arrow is an isomorphism, hence the induced map $-q$ on the cohomologies is an isomorphism, hence $q$ is unimodular.
On the other hand, $\br$ is surjective by Lemma~\ref{lemma:ns}.
Conversely, if $\br$ is surjective, then the Snake Lemma implies that the cokernels of $\chi \colon \GG \to \GG^\vee$ and $q \colon \NS(\GG) \to \NS(\GG^\vee)$ are isomorphic, so if $q$ is unimodular, so is $\chi$.
\end{proof}

The lattice $(\NS(\GG),q)$ is called {\sf Neron--Severi lattice} of a surface-like pseudolattice $\GG$.
The filtration $0 \subset \ZZ\bp \subset \bp^\perp \subset \GG$ can be thought of as an analog of the topological filtration on $\GG$.
In Corollary~\ref{corollary:serre-unipotent} below we show that the Serre operator of $\GG$ (if exists) preserves this filtration and acts on its factors as the identity.

\subsection{Canonical class}\label{subsection:canonical-class}

In this section we show how one can define the canonical class of a surface-like pseudolattice.
It is always well defined as a linear function on $\NS(\GG)$, or as an element of $\NS(\GG) \otimes \QQ$,
and under unimodularity assumption, also as an element of $\NS(\GG)$.

The rank map induces a map 
\begin{equation}\label{eq:def-lambda}
\bl \colon \bw2 \GG \to \bp^\perp, 
\qquad 
F_1 \wedge F_2 \mapsto \br(F_1)F_2 - \br(F_2)F_1.
\end{equation}
Its kernel is $\bw2(\bp^\perp)$, and if the rank map is surjective, the map $\bl$ is surjective too.

\begin{lemma}\label{lemma:canonical-class}
There is a unique element $K_\GG \in \NS(\GG) \otimes \QQ$ such that for all $v_1,v_2 \in \GG$ we have
\begin{equation}\label{eq:chi-minus-k}
\chi_-(v_1,v_2) = \chi(v_1,v_2) - \chi(v_2,v_1) = -q(K_\GG,\bl(v_1,v_2)).
\end{equation} 
If $\GG$ is unimodular, then $K_\GG \in \NS(\GG)$ is integral.
\end{lemma}
\begin{proof}
Denote by $\bar\bl \colon \bw2\GG \to \NS(\GG)$ the composition of $\bl$ with the projection $\bp^\perp \to \NS(\GG)$.
Consider the diagram
\begin{equation*}
\xymatrix{
\bw2\GG \ar@{->>}[r]^-\bl \ar[d]_{\chi_-} & \Ima (\bl) \ar@{->>}[r] \ar@{-->}[dl] & \Ima(\bar\bl) \ar@{^{(}->}[r] \ar@{..>}[dll] & \NS(\GG)
\\
\ZZ
}
\end{equation*}
Our goal is to extend the map $\chi_-$ to a linear map from $\NS(\GG)$.
First, let us show that it extends to a dashed arrow.
For this it is enough to note that $\chi_-$ vanishes on $\Ker(\bl) = \bw2(\bp^\perp)$ by Definition~\ref{def:surface-like}(3).
Next, to construct a dotted arrow it is enough to check that the dashed arrow vanishes on the kernel of the map $\Ima(\bl) \twoheadrightarrow \Ima(\bar\bl)$.
Clearly, this kernel is generated by an appropriate multiple of $\bp$.
If $v \in \GG$ is such that $\br(v) \ne 0$, then $\bl(v \wedge \bp) = \br(v)\bp$, and $\chi_-(v,\bp) = 0$ by Definition~\ref{def:surface-like}(2), hence the dashed arrow vanishes on $\br(v)\bp$.
Finally, $\Ima(\bl) \subset \bp^\perp$ is a subgroup of finite index, hence $\Ima(\bar\bl) \subset \NS(\GG)$ is a subgroup of finite index, 
hence a linear map from $\Ima(\bar\bl)$ to $\ZZ$ extends uniquely to a linear map $\NS(\GG) \to \QQ$, and since the form $q$ on $\NS(\GG)$ is nondegenerate, 
there is a unique element $K_\GG$ such that~\eqref{eq:chi-minus-k} holds.

If $\GG$ is unimodular, then $\Ima(\bl) = \bp^\perp$, $\Ima(\bar\bl) = \NS(\GG)$, hence the dashed arrow is a map $\NS(\GG) \to \ZZ$.
By unimodularity of $q$, it is given by a scalar product with an integral vector $K_\GG \in \NS(\GG)$.
\end{proof}

The sign convention in~\eqref{eq:chi-minus-k} is chosen in order to agree with the standard example.
We call the element $K_\GG \in \NS(\GG)$ {\sf the canonical class of $\GG$}. 
If $K_\GG \in \NS(\GG) \subset \NS(\GG) \otimes \QQ$, we say that the canonical class of~$\GG$ is {\sf integral}.
In terms of~\eqref{eq:chi-matrix} the canonical class is given by the formula $K_\GG\cdot v_i = (b'_i - b_i)/d$.

\subsection{Serre operator}

Assume a surface-like pseudolattice has a Serre operator $\rS_\GG$.

\begin{lemma}\label{lemma:serre-p}
Any point-like element in a surface-like pseudolattice is fixed by the Serre operator, i.e., $\rS_\GG (\bp) = \bp$.
Analogously, $\rS_\GG$ fixes the corresponding rank function, i.e. $\br (\rS_\GG(-)) = \br(-)$.
\end{lemma}
\begin{proof}
Since $\chi$ is nondegenerate, \eqref{eq:serre-symmetrize} implies $\Ker \chi_- = \Ker(1 - \rS_\GG)$.
Therefore it follows from Definition~\ref{def:surface-like}(2) that $\rS_\GG(\bp) = \bp$.
\end{proof}

It follows that the Serre operator induces an automorphism of complex~\eqref{eq:rp-p}, hence an automorphism of the Neron--Severi lattice $\NS(\GG)$.

\begin{lemma}\label{lemma:serre-identity}
The automorphism of the group $\NS(\GG)$ induced by the Serre operator is the identity map.
\end{lemma}
\begin{proof}
We have to check that $1-\rS_\GG$ acts by zero on $\NS(\GG)$, or equivalently, that it takes $\bp^\perp$ to $\ZZ\bp$.
For this note that if $v_1,v_2 \in \bp^\perp$ then
\begin{equation*}
\chi(v_1,(1-\rS_\GG)v_2) = \chi_-(v_1,v_2) = 0
\end{equation*}
by~\eqref{eq:serre-symmetrize} and Definition~\ref{def:surface-like}(3).
On the other hand, since $\chi$ is nondegenerate, its kernel on~$\bp^\perp$ is generated by $\bp$.
\end{proof}

\begin{corollary}\label{corollary:serre-unipotent}
If $\GG$ is a surface-like pseudolattice then $\rS_\GG$ is unipotent, and moreover $(1 - \rS_\GG)^3 = 0$.
\end{corollary}
\begin{proof}
Indeed, by Lemma~\ref{lemma:serre-p} the three-step filtration $0 \subset \ZZ\bp \subset \bp^\perp \subset \GG$ is fixed by $\rS_\GG$, and the induced action on its factors is the identity by Lemma~\ref{lemma:serre-identity}.
\end{proof}

The relation of the canonical class and the Serre operator is given by the following

\begin{lemma}\label{lemma:lambda-k}
For any $v \in \GG$ we have $\bl(v,\rS_\GG(v)) = \br(v)^2K_\GG \pmod \bp$.
\end{lemma}
\begin{proof}
If $\br(v) = 0$ then $\br(\rS_\GG (v)) = 0$ as well (by Lemma~\ref{lemma:serre-p}), hence both sides are zero. 
So we may assume $\br(v) \ne 0$.
Take arbitrary $v' \in \GG$ and let $r = \br(v)$, $r' = \br(v')$.
Then
\begin{multline*}
q(\bl(v,v'),\bl(v,\rS_\GG (v))) = 
-\chi(\bl(v,v'),\bl(v,\rS_\GG (v))) = 
-\chi(rv' - r'v,r \, \rS_\GG (v) - rv)) = \\ =
r^2(\chi(v',v) - \chi(v',\rS_\GG (v))) + rr'(\chi(v,\rS_\GG (v)) - \chi(v,v)) =
r^2(\chi(v',v) - \chi(v,v')) = 
r^2q(K_\GG,\bl(v,v')).
\end{multline*}
It follows that $\bl(v,\rS_\GG (v)) - r^2K_\GG$ is orthogonal to all elements of the form $\bl(v,v')$.
If $v'' \in \bp^\perp$ then $\bl(v,v' + v'') = \bl(v,v') + rv''$, hence elements of this form span all $\bp^\perp$.
It follows that $\bl(v,\rS_\GG (v)) - r^2K_\GG$ is in the kernel of $q$ on $\bp^\perp$, hence in $\QQ\bp$.
\end{proof}

\begin{corollary}
For any $v \in \GG$ we have 
$\chi(\rS_\GG (v),v) - \chi(v,v) = \br(v)^2 q(K_\GG,K_\GG)$.
\end{corollary}
\begin{proof}
Substitute $v_1 = \rS_\GG v$, $v_2 = v$ into~\eqref{eq:chi-minus-k}, take into account $\chi(v,\rS_\GG (v)) = \chi(v,v)$, use Lemma~\ref{lemma:lambda-k} and the fact that the point-like element $\bp$ is in the kernel of $q$.
\end{proof}

The (rational) number $q(K_\GG,K_\GG)$ should be thought of as the canonical degree of a surface-like pseudolattice.

\subsection{Exceptional sequences in surface-like pseudolattices}

The computations of this section are analogues of those in \cite[Section~3]{Pe}.
However, the categorical approach allows to simplify them considerably.
Assume $\GG$ is a surface-like pseudolattice, $\bp$ is its point-like element, $\br$ is the corresponding rank function, $q$ is the induced quadratic form on $\NS(\GG)$, 
and $\bl$ is the map defined in~\eqref{eq:def-lambda}.

\begin{lemma}[cf.~\protect{\cite[3.5(ii)]{Pe}}]\label{lemma:ne-pair}
Assume $\ce_1, \ce_2 \in \GG$ are exceptional with nonzero ranks.
Then
\begin{equation}
\chi(\ce_1,\ce_2) + \chi(\ce_2,\ce_1) = \frac1{\br(\ce_1)\br(\ce_2)}(q(\bl(\ce_1,\ce_2)) + \br(\ce_1)^2 + \br(\ce_2)^2).
\end{equation} 
\end{lemma}
\begin{proof}
To abbreviate the notation write $r_i = \br(\ce_i)$. 
Then $\bl(\ce_1,\ce_2) = r_1\ce_2 - r_2\ce_1$, hence
\begin{multline*}
-q(\bl(\ce_1,\ce_2)) = 
\chi(\bl(\ce_1,\ce_2),\bl(\ce_1,\ce_2)) =
\chi(r_1\ce_2 - r_2\ce_1,r_1\ce_2 - r_2\ce_1) = \\ =
r_1^2\chi(\ce_2,\ce_2) + r_2^2\chi(\ce_1,\ce_1) - r_1r_2(\chi(\ce_1,\ce_2) + \chi(\ce_2,\ce_1))
\end{multline*}
(the first is by Lemma~\ref{lemma:chi-ker-r}, the second is by definition of $\bl$, and the third is by bilinearity of $\chi$).
Using exceptionality $\chi(\ce_1,\ce_1) = \chi(\ce_2,\ce_2) = 1$, we easily deduce the required formula.
\end{proof}

\begin{lemma}[cf.~\protect{\cite[Proposition~3.8]{Pe}}]\label{lemma:ne-triple}
Assume $(\ce_1,\ce_2)$ and $(\ce_2,\ce_3)$ are exceptional pairs in $\GG$ and $\br(\ce_2) \ne 0$.
Then $(\ce_1,\ce_3)$ is an exceptional pair if and only if $q(\bl(\ce_1,\ce_2),\bl(\ce_2,\ce_3)) = \br(\ce_1)\br(\ce_3)$.
\end{lemma}
\begin{proof}
Set $r_i = \br(\ce_i)$. 
As in the previous lemma, we have
\begin{multline*}
q(\bl(\ce_1,\ce_2),\bl(\ce_2,\ce_3)) =
q(\bl(\ce_2,\ce_3),\bl(\ce_1,\ce_2)) =
-\chi(\bl(\ce_2,\ce_3),\bl(\ce_1,\ce_2)) = \\ =
-\chi(r_2\ce_3 - r_3\ce_2, r_1\ce_2 - r_2\ce_1) =
-r_1r_2\chi(\ce_3,\ce_2) - r_3r_2\chi(\ce_2,\ce_1) + r_3r_1\chi(\ce_2,\ce_2) + r_2^2\chi(\ce_3,\ce_1).
\end{multline*}
By the assumptions the first two terms in the right hand side vanish, and the third term equals $r_1r_3$. 
Hence $r_2^2\chi(\ce_3,\ce_1) = q(\bl(\ce_1,\ce_2),\bl(\ce_2,\ce_3)) - r_1r_3$.
In particular, since $r_2$ is nonzero, $\chi(\ce_3,\ce_1)$ vanishes if and only if $q(\bl(\ce_1,\ce_2),\bl(\ce_2,\ce_3)) = r_1r_3$.
\end{proof}

\begin{lemma}\label{lemma:ne-quadruple}
If $(\ce_1,\ce_2,\ce_3,\ce_4)$ is an exceptional sequence in $\GG$ then $q(\bl(\ce_1,\ce_2),\bl(\ce_3,\ce_4)) = 0$.
\end{lemma}
\begin{proof}
Set $r_i = \br(\ce_i)$. 
Then as before
\begin{multline*}
q(\bl(\ce_1,\ce_2),\bl(\ce_3,\ce_4)) =
q(\bl(\ce_3,\ce_4),\bl(\ce_1,\ce_2)) =
-\chi(\bl(\ce_3,\ce_4),\bl(\ce_1,\ce_2)) = \\ =
-\chi(r_3\ce_4 - r_4\ce_3, r_1\ce_2 - r_2\ce_1) =
-r_1r_3\chi(\ce_4,\ce_2) - r_2r_4\chi(\ce_3,\ce_1) + r_2r_3\chi(\ce_4,\ce_1) + r_1r_4\chi(\ce_3,\ce_2).
\end{multline*}
By the assumption, all the terms in the right hand side vanish.
\end{proof}

\begin{lemma}[cf.~\protect{\cite[3.3(i) and~3.4(iv)]{Pe}}]\label{lemma:ne-zero-rank}
If $\br(\ce) = 0$ then $\ce$ is exceptional in $\GG$ if and only if $q(\ce) = -1$.
If $\br(\ce_1) = \br(\ce_2) = 0$ then $\chi(\ce_1,\ce_2) = 0$ if and only if $q(\ce_1,\ce_2) = 0$.
\end{lemma}
\begin{proof}
This follows immediately from Lemma~\ref{lemma:chi-ker-r}.
\end{proof}

\section{Minimal surface-like pseudolattices}\label{section:minimal}

Let $\GG$ be a surface-like pseudolattice.
We use the theory developed in the previous section keeping the point-like element $\bp$ implicit.
Also, to simplify notation we write $D_1 \cdot D_2$ instead of $q(D_1,D_2)$ for any $D_1,D_2 \in \NS(\GG)$ and call this pairing the intersection form.
We always denote the rank of $\GG$ by $n$.

\subsection{Minimality, geometricity, and norm-minimality}

We start by introducing some useful notions.

\begin{definition}\label{def:minimlaity}
A surface-like pseudolattice $\GG$ is {\sf minimal} if it has no exceptional elements of zero rank.
\end{definition}

There is a simple characterization of minimality in terms of Neron--Severi lattices.

\begin{lemma}\label{lemma:minimality-criterion}
A surface-like pseudolattice $\GG$ is minimal if and only if the intersection form on its Neron--Severi lattice $\NS(\GG)$ does not represent $-1$.
\end{lemma}
\begin{proof}
Follows immediately from Lemma~\ref{lemma:ne-zero-rank}.
\end{proof}

\begin{definition}\label{def:geometricity}
We say that a lattice $\rL$ with a vector $K \in \rL$ is {\sf geometric} if $\rL$ has signature $(1,\rk \rL -1)$ and $K$ is {\sf characteristic}, i.e., $D^2 \equiv K\cdot D \pmod 2$ for any $D \in \rL$.
A surface-like pseudolattice $\GG$ is {\sf geometric} if its canonical class is integral and $(\NS(\GG),K_\GG)$ is a geometric lattice.
\end{definition}

\begin{remark}\label{remark:matrix}
In terms of matrix presentation~\eqref{eq:chi-matrix}, a unimodular geometric pseudolattice can be represented by a matrix with $b'_1 = \dots = b'_{n-2} = 0$, $d = 1$, 
the symmetric matrix $(c_{ij})$ representing (up to a sign) the unimodular intersection pairing of its Neron--Severi lattice, and $(b_1,\dots,b_{n-2})$ being the intersection products of the basis vectors with the anticanonical class.
Indeed, $d = \pm 1$ by the proof of Lemma~\ref{lemma:ns}, and changing the sign of $v_0$ if necessary we can ensure that $d = 1$.
Similarly, using unimodularity of $(c_{ij})$ and adding to $v_0$ an appropriate linear combination of $v_1,\dots,v_{n-2}$ we can make $b'_i = 0$.
Then $b_i = -K_\GG \cdot v_i$ by an observation at the end of Subsection~\ref{subsection:canonical-class}.
\end{remark}

The name geometric is motivated by the next result.

\begin{lemma}
A standard surface-like pseudolattice $(\GG,\bp) = (\NGr(\bD(X)),\bp_X)$ is geometric.
\end{lemma}
\begin{proof}
The canonical class in the Neron--Severi lattice $\NS(X)$ of a surface $X$ is characteristic since by Riemann--Roch we have 
$D^2 -  K_X \cdot D = 2 (\chi(\cO_X(D)) - \chi(\cO_X))$
and the signature of $\NS(X)$ is equal to $(1,\rk \NS(X) - 1)$ by Hodge index Theorem.
We conclude by $\NS(\GG) = \NS(X)$.
\end{proof}

In this paper we are mostly interested in pseudolattices with an exceptional basis.

Let $\GG$ be a surface-like pseudolattice, and $\ce_\bullet = (\ce_1,\dots,\ce_n)$ an exceptional basis in $\GG$.
We define its norm as
\begin{equation}\label{eq:norm}
||\ce_\bullet|| = \sum_{i=1}^n \br(\ce_i)^2.
\end{equation}
We say that a basis is {\sf norm-minimal}, if the norm of any exceptional basis obtained from $\ce_\bullet$ by a sequence of mutations is greater or equal than the norm of $\ce_\bullet$.
Clearly, any exceptional basis can be transformed by mutations to a norm-minimal basis.

The main result of this section is the following

\begin{theorem}\label{theorem:minimal-cats}
Assume $\GG$ is a geometric surface-like pseudolattice of rank $n$ and $\ce_\bullet$ is a norm-minimal exceptional basis in $\GG$.
If\/ $\br(\ce_i) \ne 0$ for all $1 \le i \le n$ \textup{(}for instance if $\GG$ is minimal\textup{)}, then $\GG$ is isometric either to~$\NGr(\bD(\PP^2))$ or to~$\NGr(\bD(\PP^1 \times \PP^1))$, 
and the basis $\ce_\bullet$ corresponds to one of the standard exceptional collections of line bundles in these categories \textup{(}see Example~\textup{\ref{example:p2-p1p1}}\textup{)}.
In particular, $n = 3$ or~$n = 4$, $K_\GG^2 = 12 - n$, and $\br(\ce_i) = \pm 1$ for all $i$.
\end{theorem}

The proof takes the rest of this section, and uses essentially Perling's arguments.
It is obtained as a combination of Corollaries~\ref{corollary:minimal-n}, \ref{corollary:minimal-3}, and~\ref{corollary:minimal-4}.
Note that $n \ge 3$ by geometricity, since the rank of $\NS(\GG)$ is at least 1 as its signature is equal to $(1,n-3)$.
We will implicitly assume this inequality from now on.

\subsection{Toric system associated with an exceptional collection}

We start with a general important definition, which is a small modification of Perling's definition.
The history, in fact, goes back to the work~\cite{HP} where the authors associated a fan of a smooth toric surface to an exceptional collection of line bundles on any rational surface. 
The notion below is used to generalize this construction to exceptional collections consisting of objects of arbitrary non-zero ranks, and is a small modification of the one, introduced by Perling in~\cite{Pe}.

\begin{definition}[cf.~\protect{\cite[Definition~5.5]{Pe}}]
\label{def:toric}
Let $(\rL,K)$ be a lattice of rank $n - 2$ with a vector $K \in \rL$. 
Let $\lambda_{i,i+1} \in \rL \otimes \QQ$, $1 \le i \le n$, be a collection of $n$ vectors.
Put $\lambda_{i+n,i+n+1} = \lambda_{i,i + 1}$ for all $i \in \ZZ$ and
\begin{equation*}
\lambda_{i,j} = \lambda_{i,i+1} + \dots + \lambda_{j-1,j}
\qquad\text{for all $i < j < i + n$}.
\end{equation*}
We say that the collection $\lambda_{\bullet,\bullet}$ is a {\sf toric system} in $\rL$, if
\begin{enumerate}
\item 
there exist integers $r_i \in \ZZ$ such that for all $i \in \ZZ$ we have
\begin{equation*}
\lambda_{i-1,i} \cdot \lambda_{i,i+1} = \frac{1}{r_i^2}; 
\end{equation*}
\item 
for all $i < j < i + n - 2$ we have $\lambda_{i-1,i} \cdot \lambda_{j,j+1} = 0$;
\item 
for all $i < j < i + n$ the vectors $r_i r_j \lambda_{i,j}$ are integral, hence their squares are integers, i.e.,
\begin{equation*}
r_i r_j \lambda_{i,j} \in \rL
\qquad\text{and}\qquad 
a_{i,j} := (r_ir_j \lambda_{i,j})^2 \in \ZZ;
\end{equation*}
\item 
for all $i < j < i + n$ the fraction 
\begin{equation*}
n_{i,j} := \frac{a_{i,j} + r_i^2 + r_j^2}{r_ir_j} \in \ZZ
\end{equation*}
is an integer; 
\item 
$\lambda_{1,2} + \dots + \lambda_{n,n+1} = -K$; and finally
\item 
$\gcd(r_1,\dots,r_n) = 1$.
\end{enumerate}
\end{definition}

Note that if $\lambda_{\bullet,\bullet}$ is a toric system, the integers $r_i$ are determined up to a sign, and whether the other conditions hold or not does not depend on the choice of these signs.
Furthermore, if the signs of $r_i$ are chosen, the other integers $a_{i,j}$ and $n_{i,j}$ are determined unambiguously. 

\begin{remark}
The main difference between Perling's definition and Definition~\ref{def:toric} is that we demand integrality of $r_ir_j\lambda_{i,j}$ and of $n_{i,j}$ for all $i < j < i + n$, while Perling does that only for $j = i + 1$.
\end{remark}

For reader's convenience we write down the Gram matrix of the scalar product on $\rL \otimes \QQ$ on the set~$\lambda_{i,i+1}$ for $0 \le i \le n-1$ 
(note, however, that this set is not a basis in $\NS(X)_\QQ$, since $n > \dim \rL \otimes \QQ$, so the matrix below is degenerate):
\begin{equation}\label{eq:matrix-lambda}
(\lambda_{i,i+1} \cdot \lambda_{j,j+1}) = 
\begin{pmatrix}
\frac{a_{0,1}}{r_0^2r_1^2} & \frac1{r_1^2} & 0 & 0 & \dots & 0 & \frac1{r_0^2}
\\
\frac1{r_1^2} & \frac{a_{1,2}}{r_1^2r_2^2} & \frac1{r_2^2} & 0 & \dots & 0 & 0
\\
0 & \frac1{r_2^2} & \frac{a_{2,3}}{r_2^2r_3^2} & \frac1{r_3^2} & \dots & 0 & 0
\\
0 & 0 & \frac1{r_3^2} & \frac{a_{2,3}}{r_2^2r_3^2} & \dots & 0 & 0
\\
\vdots & \vdots & \vdots & \vdots & \ddots & \vdots & \vdots 
\\
0 & 0 & 0 & 0 & \dots & \frac{a_{n-2,n-1}}{r_{n-2}^2r_{n-1}^2} & \frac1{r_{n-1}^2}
\\
\frac1{r_n^2} & 0 & 0 & 0 & \dots & \frac1{r_{n-1}^2} & \frac{a_{n-1,n}}{r_{n-1}^2r_n^2} 
\end{pmatrix}
\end{equation}
Note that the matrix is cyclically tridiagonal and symmetric (since $r_n = r_0$ by periodicity of $\lambda_{\bullet,\bullet}$).
The submatrix of~\eqref{eq:matrix-lambda} obtained by deleting the last row and column is used extensively in~\cite{V}.

Recall the map $\bl \colon \bw2\GG \to \bp^\perp$ defined in~\eqref{eq:def-lambda}.
We implicitly compose it with $\bp^\perp \to \bp^\perp/\bp = \NS(\GG)$.

\begin{proposition}\label{proposition:lambda-toric}
Let $\ce_1,\dots,\ce_n$ be an exceptional basis in a surface-like pseudolattice $\GG$ with $\br(\ce_i) \ne 0$ for all $i$.
Extend it to an infinite sequence of vectors in $\GG$ by setting $\ce_{i + n} = \rS_\GG^{-1}(\ce_i)$ for all $i \in \ZZ$, where~$\rS_\GG$ is the Serre operator.
Then the collection of vectors
\begin{equation}\label{eq:def-r-lambda}
\lambda_{i,i+1} := \frac1{\br(\ce_i)\br(\ce_{i+1})} \bl(\ce_i,\ce_{i+1}) = \frac{\ce_{i+1}}{\br(\ce_{i+1})} - \frac{\ce_{i}}{\br(\ce_{i})} \in \NS(\GG) \otimes \QQ,
\qquad\qquad\text{$1 \le i \le n$,}
\end{equation}
is a toric system in $(\NS(\GG),K_\GG)$ with $r_i = \br(\ce_i)$.
\end{proposition}

Note that the equality $r_i = \br(\ce_i)$ is one of the results of the theorem --- we claim that the integers $r_i$ defined from the toric sequence agree (of course, up to a sign) with the ranks of the original vectors $\ce_i$.

\begin{proof}
Set $r_i := \br(\ce_i)$.
Then equations in Definition~\ref{def:toric}(1) follow from Lemma~\ref{lemma:ne-triple}, and those in Definition~\ref{def:toric}(2) from Lemma~\ref{lemma:ne-quadruple}.
Furthermore, we have 
\begin{equation*}
\lambda_{i,j} = \lambda_{i,i+1} + \dots + \lambda_{j-1,j} =
\left(\frac{\ce_{i+1}}{r_{i+1}} - \frac{\ce_{i}}{r_{i}}\right) + \dots + \left(\frac{\ce_{j}}{r_{j}} - \frac{\ce_{j-1}}{r_{j-1}}\right) = 
\frac{\ce_{j}}{r_{j}} - \frac{\ce_{i}}{r_{i}} = 
\frac1{r_ir_j} \bl(\ce_i,\ce_j),
\end{equation*}
hence $r_ir_j\lambda_{i,j} = \bl(\ce_i,\ce_j)$ is integral; this proves Definition~\ref{def:toric}(3).
By Lemma~\ref{lemma:ne-pair} we have
\begin{equation*}
n_{i,j} = \chi(\ce_i,\ce_j)
\end{equation*}
as soon as $i < j < i + n$; this proves Definition~\ref{def:toric}(4). 
Furthermore, 
\begin{equation*}
\sum_{i=1}^n \lambda_{i,i+1} = \bl(\ce_1,\ce_{n+1})/r_1r_{n+1} = \bl(\rS_\GG(\ce_{n+1}),\ce_{n+1})/r_{n+1}^2, 
\end{equation*}
and by Lemma~\ref{lemma:lambda-k} in $\NS(\GG)$ this is equal to $-K_\GG$; this proves Definition~\ref{def:toric}(5).
Finally, Definition~\ref{def:toric}(6) follows from the fact that $\ce_i$ form a basis of $\GG$, and the rank map is surjective since $\GG$ is unimodular.
\end{proof}

\begin{remark}\label{remark:toric-mutations}
One can define mutations of toric systems in a way compatible with mutations of exceptional collections.
However, we will not need this in our paper, so we skip the construction.
\end{remark}

Below we discuss some properties of toric systems.
We always denote by $r_i$, $a_{i,j}$, and~$n_{i,j}$ the integers determined by the toric system (with some choice of signs of $r_i$).
We consider the index set of a toric system as $\ZZ/n\ZZ$ with its natural cyclic order.
We say that indices $i$ and $j$ are {\sf cyclically distinct} if $i - j \ne 0 \pmod n$, and {\sf adjacent} if $i - j = \pm 1 \pmod n$.

\begin{lemma}\label{lemma:chains-independent}
If $\lambda_{\bullet,\bullet}$ is a toric system, then for any $i \in \ZZ$ and for any $0 \le k \le n-2$ the sequence of~$k$ vectors $(\lambda_{i,i+1},\dots,\lambda_{i+k-1,i+k})$ 
is linearly independent.
In particular, $\lambda_{i,i+1} \ne 0$ for all $i$.
\end{lemma}
\begin{proof}
We prove by induction on $k$.
When $k = 0$ there is nothing to prove.
Assume now that $k \ge 1$ and the claim for $k - 1$ is proved.
Assume that $\lambda_{i+k,i+k+1} = x_i\lambda_{i,i+1} + \dots + x_{i+k-1}\lambda_{i+k-1,i+k}$.
Consider the scalar product of this equality with $\lambda_{i+k+1,i+k+2}$.
The left hand side equals $1/r_{i+k+1}^2 \ne 0$ by Definition~\ref{def:toric}(1), while the right hand side is zero by Definition~\ref{def:toric}(2).
This contradiction proves the step.
\end{proof}

The next lemma uses the signature assumption on the Neron--Severi lattice.

\begin{lemma}\label{lemma:extremal-pair}
Let $\lambda_{\bullet,\bullet}$ be a toric system in a lattice $\rL$ of signature $(1,n-3)$.
If both $a_{i,i+1} \ge 0$ and~$a_{j,j+1} \ge 0$ \textup{(}i.e., $\lambda_{i,i+1}^2 \ge 0$ and $\lambda_{j,j+1}^2 \ge 0$\textup{)} for cyclically distinct $i$ and $j$ then 
\begin{itemize}
\item 
either $i$ and $j$ are adjacent, 
\item 
or $n = 4$, $j = i + 2 \pmod n$, $\lambda_{i,i+1}$ and $\lambda_{j,j+1}$ are proportional, and $a_{i,i+1} = a_{j,j+1} = 0$.
\end{itemize}
\end{lemma}
\begin{proof}
We may assume that $i < j < i + n$.
Assume $i$ and $j$ are not adjacent. 
Then the intersection form on the sublattice of $\rL_\QQ$ spanned by $\lambda_{i,i+1}$ and $\lambda_{j,j+1}$ looks as
\begin{equation*}
\left(
\begin{smallmatrix}
a_{i,i+1} & 0 \\ 0 & a_{j,j+1}
\end{smallmatrix}
\right),
\end{equation*}
hence is non-negatively determined. 
But by the signature assumption $\rL_\QQ$ does not contain non-negatively determined sublattices of rank 2, hence the vectors $\lambda_{i,i+1}$ and $\lambda_{j,j+1}$ are proportional.
By Lemma~\ref{lemma:chains-independent} it follows that $j -i \ge n - 2$ and $(i + n) - j \ge n - 2$. 
Summing up, we deduce $n \ge 2n - 4$, hence $n \le 4$.
Since for $n = 3$ any two cyclically distinct $i$ and $j$ are adjacent, we conclude that $n = 4$, and non-adjacency means $j = i + 2$.
Finally, $\lambda_{i,i+1} \cdot \lambda_{j,j+1}= 0$ by Definition~\ref{def:toric}(2), and since the vectors are proportional and nonzero, we also have $\lambda_{i,i+1}^2 = \lambda_{j,j+1}^2 = 0$, hence $a_{i,i+1} = a_{j,j+1} = 0$.
\end{proof}

\subsection{Fan associated with a toric system}

Following Perling, we associate with a toric system $\lambda_{\bullet,\bullet}$ in the Neron--Severi lattice $\NS(\GG)$ of a surface-like pseudolattice $\GG$ a fan in~$\ZZ^2$.
As it is customary in toric geometry, we consider a pair of mutually dual free abelian groups
\begin{equation*}
M \cong \ZZ^2
\qquad\text{and}\qquad 
N := M^\vee.
\end{equation*}
We define a map $M \to \ZZ^n$ as the kernel of the map $\ZZ^n \to \NS(\GG)_\QQ$ defined by taking the base vectors to~$\lambda_{i,i+1}$. 
Then we consider the dual map $(\ZZ^n)^\vee \to N$ and denote the images of the base vectors by~$\ell_{i,i+1} \in N$.
The definition of $\ell_{i,i+1}$ implies that
\begin{equation}\label{eq:ell-relation-tensor}
\sum_{i=1}^n \ell_{i,i+1} \otimes \lambda_{i,i+1} = 0
\qquad\text{in $N \otimes \NS(\GG)_\QQ$},
\qquad\text{and}\qquad
\text{$N$ is generated by $\ell_{i,i+1}$}
\end{equation}
(the toric system $(\lambda_{i,i+1})$ defines an element of $\NS(\GG)_\QQ \otimes (\ZZ^n)^\vee \cong \Hom(\ZZ^n, \NS(\GG)_\QQ)$, 
while the collection of vectors~$(\ell_{i,i+1})$ defines an element of $\ZZ^n \otimes N \cong \Hom(M,\ZZ^n)$, and the sum in~\eqref{eq:ell-relation-tensor} is an expression for the composition of these maps).

\begin{remark}
Perling considers the fan in $N_\RR$ generated by the vectors $\ell_{i,i+1}$, proves that it defines a projective toric surface (\cite[Proposition~10.6]{Pe}), and shows some nice results about it.
For instance, he proves that its singularities are T-singularities (which gives a nice connection to Hacking's results~\cite{Ha}),
and relates mutations of exceptional collections and their toric systems (see Remark~\ref{remark:toric-mutations}) to degenerations of toric surfaces.
We, however, will not need this material and suggest the interested reader to look into~\cite{Pe}.
So, toric geometry will not be explicitly discussed below, but an experienced reader will notice it always lurking at the background.
\end{remark}

The general tensor relation~\eqref{eq:ell-relation-tensor} implies many linear relations.

\begin{proposition}\label{proposition:ell-relation}
For every $i \in \ZZ$ we have
\begin{equation}\label{eq:ell-relation}
r_{i+1}^2\ell_{i-1,i} + a_{i,i+1}\ell_{i,i+1} + r_{i}^2\ell_{i+1,i+2} = 0.
\end{equation}
\end{proposition}
\begin{proof}
Take the scalar product of~\eqref{eq:ell-relation-tensor} with $r_i^2r_{i+1}^2\lambda_{i,i+1}$ and use Definition~\ref{def:toric}.
\end{proof}

\begin{corollary}\label{corollary:l-independent}
For every $i \in \ZZ$ the vectors $\ell_{i-1,i}$ and $\ell_{i,i+1}$ are linearly independent.
\end{corollary}
\begin{proof}
Assume $\ell_{i-1,i}$ and $\ell_{i,i+1}$ are linearly dependent.
Then there is a nonzero element $m \in M$ such that $m(\ell_{i-1,i}) = m(\ell_{i,i+1}) = 0$.
Evaluating it on~\eqref{eq:ell-relation-tensor}, we see that 
\begin{equation*}
m(\ell_{i+1,i+2})\lambda_{i+1,i+2} + \dots + m(\ell_{i+n-2,i+n-1})\lambda_{i+n-2,i+n-1} = 0.
\end{equation*}
But by Lemma~\ref{lemma:chains-independent} the vectors $\lambda_{i+1,i+2},\dots,\lambda_{i+n-2,i+n-1}$ are linearly independent.
It follows that $m(\ell_{j,j+1}) = 0$ for all $j$.
This contradicts with the fact that $\ell_{j,j+1}$ generate $N$.
\end{proof}

On the other hand, we have relations of a completely different sort.
Denote by $\det \colon N \times N \to \ZZ$ a skew-symmetric bilinear form on $N$ which induces an isomorphism $\bw2N \xrightarrow{\ \sim\ } \ZZ$ (i.e., a volume form).

\begin{proposition}\label{proposition:ell-relation-det}
For every $i \in \ZZ$ we have
\begin{equation}\label{eq:ell-relation-dets}
\det(\ell_{i,i+1},\ell_{i+1,i+2}) \ell_{i-1,i} + \det(\ell_{i+1,i+2},\ell_{i-1,i}) \ell_{i,i+1} + \det(\ell_{i-1,i},\ell_{i,i+1}) \ell_{i+1,i+2} = 0.
\end{equation}
\end{proposition}
\begin{proof}
This is standard linear algebra.
\end{proof}

With appropriate choice of the volume form, relations~\eqref{eq:ell-relation} almost coincide with relations~\eqref{eq:ell-relation-dets}.

\begin{proposition}\label{proposition:r-a-det}
There is a choice of a volume form $\det$ on $N$ and a positive integer $h \in \ZZ$ such that
\begin{equation}\label{eq:r-dets}
\det(\ell_{i-1,i},\ell_{i,i+1}) = hr_i^2,
\quad\text{and}\quad 
\det(\ell_{i+1,i+2},\ell_{i-1,i}) = ha_{i,i+1}
\quad\text{for all $i \in \ZZ$.}
\end{equation}
With this choice we have $\det(\ell_{i-1,i},\ell_{i,i+1}) > 0$ for all $i \in \ZZ$.
\end{proposition}
\begin{proof}
Indeed, by Corollary~\ref{corollary:l-independent} the space of relations between $\ell_{i-1,i}$, $\ell_{i,i+1}$, $\ell_{i+1,i+2}$ is one-dimensional.
Since both relations~\eqref{eq:ell-relation} and~\eqref{eq:ell-relation-dets} are nontrivial (the first because $r_i \ne 0$ and the second because $\det(\ell_{i-1,i},\ell_{i,i+1}) \ne 0$), 
they are proportional, hence for every $i \in \ZZ$ there is unique $h_i \in \QQ \setminus 0$ such that
\begin{equation*}
\det(\ell_{i-1,i},\ell_{i,i+1}) = h_ir_i^2,
\qquad 
\det(\ell_{i,i+1},\ell_{i+1,i+2}) = h_ir_{i+1}^2,
\quad\text{and}\quad 
\det(\ell_{i+1,i+2},\ell_{i-1,i}) = h_ia_{i,i+1}.
\end{equation*}
Comparing the relations for $i$ and $i+1$, we see that $h_i = h_{i+1}$. 
Hence $h_i = h$ for one non-zero rational number $h$, so that~\eqref{eq:r-dets} holds.
Furthermore, since all $r_i$ are mutually coprime (by Definition~\ref{def:toric}(6)), it follows that $h$ is a non-zero integer.
Finally, changing the volume form on $N$ if necessary, we can assume that $h > 0$.
\end{proof}

\begin{remark}
In fact, Perling claims that $h = 1$ (see~\cite[Proposition~8.2]{Pe}), however, his proof of this fact is unclear. 
On the other hand, this is not necessary for the proof of the main result.
\end{remark}

Now let us deduce some consequences about the geometry of vectors $\ell_{i,i+1}$ on the plane $N_\RR \cong \RR^2$.
Consider the polygon defined as the convex hull of the vectors $\ell_{i,i+1}$:
\begin{equation*}
\bP := \Conv(\ell_{1,2},\ell_{2,3},\dots,\ell_{n,n+1}) \subset N_\RR.
\end{equation*}

\begin{lemma}\label{lemma:polygon-zero}
The point $0 \in N$ is contained in the interior of the polygon $\bP$.
\end{lemma}
\begin{proof}
Indeed, otherwise all $\ell_{i,i+1}$ are contained in a closed half-plane of $N_\RR$.
On the other hand, by Proposition~\ref{proposition:r-a-det} we have $\det(\ell_{i-1,i},\ell_{i,i+1}) > 0$ for all $i \in \ZZ$, hence the oriented angle between the vectors $\ell_{i-1,i}$ and $\ell_{i,i+1}$ is contained in the interval $(0,\pi)$.
Evidently, a periodic sequence of vectors with this property cannot be contained in a half-plane.
This contradiction proves the lemma.
\end{proof}

We say that a vector $\ell_{i,i+1}$ is {\sf extremal} if it is a vertex of the polygon $\bP$.
In other words, if it does not lie in the convex hull of the other vectors.

\begin{corollary}\label{corollary:three-extremal}
There are at least three extremal vectors among $\ell_{i,i+1}$.
\end{corollary}
\begin{proof}
Indeed, the polygon $\bP$ has a nonempty interior by Lemma~\ref{lemma:polygon-zero}, hence it has at least three vertices.
\end{proof}

\subsection{Fan of a norm-minimal basis}

Throughout this section we assume $\lambda_{\bullet,\bullet}$ is the toric system of a norm-minimal exceptional basis $\ce_\bullet$ in a surface-like pseudolattice $\GG$,
i.e., the sum $\sum_{i=1}^n r_i^2$ of the squares of the ranks of the basis vectors is minimal possible among all mutations of the basis.

\begin{lemma}\label{lemma:a-r}
Assume the basis is norm-minimal.
Then for every $i \in \ZZ$ we have
\begin{itemize}
\item 
if $a_{i,i+1} \ge 0$ then $a_{i,i+1} \ge |r_i^2 - r_{i+1}^2|$, and
\item 
if $a_{i,i+1} < 0$ then $a_{i,i+1} \le -(r_i^2 + r_{i+1}^2)$.
\end{itemize}
\end{lemma}
\begin{proof}
We have $\chi(\ce_i,\ce_{i+1}) = n_{i,i+1}$ (see the proof of Proposition~\ref{proposition:lambda-toric}), hence by~\eqref{eq:mutations} the rank of the left mutation $\LL_{\ce_i}(\ce_{i+1})$ of $\ce_{i+1}$ through $\ce_i$ equals
\begin{equation*}
|r'| = |n_{i,i+1}r_i - r_{i+1}| = \frac{|a_{i,i+1}+r_i^2|}{|r_{i+1}|}.
\end{equation*}
Norm-minimality implies $|r'| \ge |r_{i+1}|$, i.e.
\begin{equation*}
|a_{i,i+1} + r_i^2| \ge r_{i+1}^2.
\end{equation*}
Considering analogously the right mutation $\RR_{\ce_{i+1}}(\ce_{i})$, we deduce 
\begin{equation*}
|a_{i,i+1} + r_{i+1}^2| \ge r_i^2.
\end{equation*}
Analyzing the cases of nonnegative and negative $a_{i,i+1}$, we easily deduce the required inequalities.
\end{proof}

Note that in the proof we only use a local norm-minimality of the basis, 
i.e., that the norm of the basis does not decrease under elementary mutations only.

\begin{lemma}\label{lemma:negative-not-extremal}
Assume the basis is norm-minimal. 
If $a_{i,i+1} < 0$ then $\ell_{i,i+1} \in \Conv(0,\ell_{i-1,i},\ell_{i+1,i+2})$.
In particular, $\ell_{i,i+1}$ is not extremal.
\end{lemma}
\begin{proof}
If $a_{i,i+1} < 0$ then the relation~\eqref{eq:ell-relation} can be rewritten as
\begin{equation*}
\ell_{i,i+1} = \frac{r_i^2}{|a_{i,i+1}|}\ell_{i-1,i} + \frac{r_{i+1}^2}{|a_{i,i+1}|}\ell_{i+1,i+2}
\end{equation*}
By Lemma~\ref{lemma:a-r} we have $|a_{i,i+1}| \ge r_i^2 + r_{i+1}^2$, hence the coefficients in the right hand side are nonnegative and their sum does not exceed 1, hence the claim about the convex hull.

Now, since $0$ is in the interior of $\bP$ (Lemma~\ref{lemma:polygon-zero}), we see that $\Conv(0,\ell_{i-1,i},\ell_{i+1,i+2}) \subset \bP$, so the only possibility for~$\ell_{i,i+1}$ to be extremal is 
if it coincides with one of the vertices of the triangle $\Conv(0,\ell_{i-1,i},\ell_{i+1,i+2})$. But this is impossible by Corollary~\ref{corollary:l-independent}.
\end{proof}

Combining the last two lemmas we see that there are at least three indices $i$ such that $a_{i,i+1} \ge 0$.
This already gives the required restriction on $n$.

\begin{corollary}\label{corollary:minimal-n}
If a norm-minimal exceptional basis in a geometric surface-like pseudolattice $\GG$
consists of elements of non-zero rank, then $n = 3$ or $n = 4$.
\end{corollary}
\begin{proof}
Each vertex of the polygon $\bP$ corresponds to an extremal vector~$\ell_{i,i+1}$, hence the corresponding integer $a_{i,i+1}$ is nonnegative by Lemma~\ref{lemma:negative-not-extremal}.
Thus, if there is a pair of non-adjacent $i$ and $j$ such that the vectors $\ell_{i,i+1}$ and $\ell_{j,j+1}$ are extremal, then $n = 4$ by Lemma~\ref{lemma:extremal-pair}.
On the other hand, if all $i$ such that $\ell_{i,i+1}$ is extremal are pairwise adjacent, then clearly $n = 3$.
\end{proof}

\subsection{Norm-minimal bases for $n = 3$ and $n = 4$}

It remains to consider two cases.
As before we assume that~$\ce_\bullet$ is a norm-minimal exceptional basis of a geometric pseudolattice $\GG$ of rank $n$ 
consisting of elements of non-zero rank and $\lambda_{\bullet,\bullet}$ is its toric system constructed in Proposition~\ref{proposition:lambda-toric}.

\begin{lemma}\label{lemma:n=3}
If $n = 3$, then $r_i = \pm 1$, $a_{i,i+1} = 1$, and $K_\GG^2 = 9$.
\end{lemma}
\begin{proof}
Replacing $\ce_i$ by $-\ce_i$, we may assume that all the ranks are positive.
Let $H$ denote a generator of~$\NS(\GG)$,
by unimodularity (Lemma~\ref{lemma:chi-unimodular}) and geometricity of $\GG$, we have $H^2 = 1$. 
By definition of a toric system
$\lambda_{ij} = c_{ij}H$, $c_{ij} \in \QQ$. 
Accordingly, we have by~Definition~\ref{def:toric}(1)
\begin{equation*}
c_{12}c_{31} = \frac1{r_1^2},
\qquad 
c_{12}c_{23} = \frac1{r_2^2},
\qquad 
c_{31}c_{23} = \frac1{r_3^2}.
\end{equation*}
Solving this system for $c_{ij}$, we see that 
\begin{equation*}
c_{12} = \frac{r_3}{r_1r_2},\qquad
c_{31} = \frac{r_2}{r_3r_1},\qquad
c_{23} = \frac{r_1}{r_2r_3}
\end{equation*}
up to a common sign, which can be fixed by replacing $H$ with $-H$ if necessary.
By Definition~\ref{def:toric}(5)
\begin{equation*}
-K_\GG = \lambda_{12} + \lambda_{23} + \lambda_{31} = 
\left(\frac{r_3}{r_1r_2} + \frac{r_2}{r_3r_1} + \frac{r_1}{r_2r_3}\right)H = \frac{r_1^2 + r_2^2 + r_3^2}{r_1r_2r_3}H.
\end{equation*}
Since $\GG$ is unimodular, $K_\GG$ is integral by Lemma~\ref{lemma:canonical-class}, i.e., there is $\gamma \in \ZZ$ such that $-K_\GG = \gamma H$ and hence
\begin{equation}\label{eq:markov}
r_1^2  +r_2^2 + r_3^2 = \gamma r_1r_2r_3.
\end{equation}
Since all $r_i$ are positive, so is $\gamma$.
Moreover, $\gcd(r_1,r_2,r_3) = 1$ by Definition~\ref{def:toric}(6) and Proposition~\ref{proposition:lambda-toric}.
But the only positive integer~$\gamma$ for which the above equation has an integral indivisible solution is $\gamma = 3$ (\cite[\S2.1]{A}).
Therefore, we have $K_\GG^2 = \gamma^2 = 9$.

Furthermore, in case $\gamma = 3$, equation~\eqref{eq:markov} is the Markov equation, and its positive norm-minimal solution (with respect to the standard braid group action) is $r_1 = r_2 = r_3 = 1$.
We obtain $c_{i,i+1} = 1$ and $a_{i,i+1} = r_i^2r_{i+1}^2c_{i,i+1}^2 = 1$.
\end{proof}

Thus, the Neron--Severi lattice of a geometric surface-like pseudolattice with an exceptional basis of length $n = 3$ can be written as
\begin{equation}\label{eq:ns-3}
\NS(\GG) = \ZZ H,\qquad 
H^2 = 1,\qquad
K_\GG = -3H,
\end{equation}
the toric system of a norm-minimal exceptional collection in $\GG$ is $\lambda_{1,2} = \lambda_{2,3} = \lambda_{1,3} = H$.
Since $\chi(\ce_i,\ce_j) = n_{i,j} = (a_{i,j} + r_i^2  + r_j^2)/r_ir_j$ and $a_{i,j} = (r_ir_j\lambda_{i,j})^2$, we see that the Gram matrix of the form $\chi$ in the basis $\ce_i$ 
is equal to the form $\chi_{\PP^2}$ from Example~\ref{example:p2-p1p1}.

\begin{corollary}\label{corollary:minimal-3}
If a geometric surface-like pseudolattice $\GG$ of rank $n = 3$  has an exceptional basis, 
then~$\GG$ is isometric to~$\NGr(\bD(\PP^2))$, and its norm-minimal basis corresponds to the exceptional collection $(\cO,\cO(1),\cO(2))$ in $\bD(\PP^2)$.
\end{corollary}

Now we pass to the case $n = 4$.
The results in this case are quite close to the results of~\cite{dTVdB}.

\begin{lemma}\label{lemma:n=4}
If $n = 4$, then $r_i = \pm 1$ and $K_\GG^2 = 8$.
\end{lemma}
\begin{proof}
Again, we may assume that all the ranks are positive.
By Corollary~\ref{corollary:three-extremal} there are at least three extremal vectors $\ell_{i,i+1}$, and by Lemma~\ref{lemma:negative-not-extremal} the corresponding integers $a_{i,i+1}$ are non-negative.
After a cyclic shift of indices we may assume 
\begin{equation*}
a_{12}, a_{23}, a_{34} \ge 0.
\end{equation*}
Applying Lemma~\ref{lemma:extremal-pair} we conclude that $a_{12} = a_{34} = 0$, $\lambda_{12}$ and $\lambda_{34}$ are proportional, and $\lambda_{12}^2 = \lambda_{34}^2 = 0$.
By Lemma~\ref{lemma:a-r} we have
\begin{equation*}
r_1^2 = r_2^2,
\qquad 
r_3^2 = r_4^2.
\end{equation*}
Since a multiple of $\lambda_{12}$ is integral, we have
\begin{equation*}
\lambda_{12} = c_{12}f,
\qquad 
\lambda_{34} = c_{34}f,
\qquad 
\text{where $f \in \NS(\GG)$ is primitive with $f^2 = 0$}
\end{equation*}
with $c_{12}, c_{34} \in \QQ$.
Since $\NS(\GG)$ is unimodular, there is an element $s \in \NS(\GG)$ such that $(f,s)$ is a basis of $\NS(\GG)$ and 
\begin{equation*}
f \cdot s = 1,
\qquad 
d := s^2 \in \{0, -1 \}.
\end{equation*}
We have
\begin{equation*}
\lambda_{23} = c_{23}s + c'_{23}f,
\qquad 
\lambda_{41} = c_{41}s + c'_{41}f.
\end{equation*}
Here all $c$ and $c'$ are rational numbers, and by Lemma~\ref{lemma:chains-independent} all $c$ are nonzero.
We have
\begin{equation*}
c_{23}c_{34} = \lambda_{23} \cdot \lambda_{34} = \frac1{r_3^2} = \frac1{r_4^2} = \lambda_{34} \cdot \lambda_{41} = c_{34}c_{41},
\end{equation*}
hence $c_{23} = c_{41}$.
Further,
\begin{equation*}
-K_\GG = \lambda_{12} + \lambda_{23} + \lambda_{34} + \lambda_{41} = (c_{12} + c'_{23} + c_{34} + c'_{41})f + 2c_{23}s.
\end{equation*}
Since $K_\GG$ is characteristic and $f^2 = 0$, we have $K_\GG \cdot f = 2c_{23}$ is even, hence $c_{23} \in \ZZ$.
On the other hand, $r_1r_2\lambda_{12} = r_2^2c_{12}f$ and $r_3r_4\lambda_{34} = r_3^2c_{34}f$ are integral, hence $r_2^2c_{12}$ and $r_3^2c_{34}$ are both integral.
But
\begin{equation*}
\frac1{r_2^2} = \lambda_{12}\lambda_{23} = c_{12}c_{23}
\qquad\text{and}\qquad
\frac1{r_3^2} = \lambda_{23}\lambda_{34} = c_{23}c_{34},
\end{equation*}
hence $(r_2^2c_{12})c_{23} = (r_3^2c_{34})c_{23} = 1$, hence (changing the signs of $f$ and $s$ if necessary) we can write
\begin{equation*}
c_{12} = \frac1{r_2^2}, 
\qquad
c_{34} = \frac1{r_3^2},
\qquad
c_{23} = c_{41} = 1.
\end{equation*}
Finally, from $0 = \lambda_{23}\cdot \lambda_{41} = d + c'_{23} + c'_{41}$ we deduce that $c'_{23} + c'_{41} = -d \in \ZZ$, hence from integrality of~$K_\GG$ we obtain
\begin{equation*}
\frac1{r_2^2} + \frac1{r_3^2} = c_{12} + c_{34} \in \ZZ,
\end{equation*}
and from this it easily follows that $r_2^2 = r_3^2 = 1$.
We finally see that all $r_i$ are equal to $1$, all $c_{i,i+1} = 1$, and $-K_\GG = (2-d)f + 2s$, hence $K_\GG^2 = 4(2-d) + 4d = 8$.
\end{proof}

Thus, the Neron--Severi lattice of a geometric surface-like pseudolattice of rank $n = 4$ with an exceptional basis can be written as
\begin{equation}\label{eq:ns-4}
\NS(\GG) = \ZZ f \oplus \ZZ s,\qquad 
f^2 = 0,\quad f\cdot s = 1, \quad s^2  = d,\qquad
K_\GG = (2-d)f - 2s
\end{equation}
and $d \in \{0, -1\}$.
 
Assume first $d = 0$.
As it was shown in the proof of Lemma~\ref{lemma:n=4}, the toric system of a norm-minimal exceptional basis in $\GG$ has form
$\lambda_{12} = f$, $\lambda_{23} = s + c'f$, $\lambda_{34} = f$ for some $c' \in \ZZ$.
This allows to compute all $a_{i,j}$ and~$n_{i,j}$, and to show that the Gram matrix of the form $\chi$
is equal to the form $\chi_{\PP^1 \times \PP^1}$ from Example~\ref{example:p2-p1p1} with $c = c' + 1$.

Similarly, if $d = -1$ the toric system of a norm-minimal exceptional basis in $\GG$ should have form
$\lambda_{12} = f$, $\lambda_{23} = s + c'f$, $\lambda_{34} = f$ for some $c' \in \ZZ$, and computing integers $n_{i,j}$ 
one checks that $\GG$ is isometric to $\NGr(\bD(\mathbb{F}_1))$, where $\mathbb{F}_1$ is the Hirzebruch surface, see Example~\ref{example:p2-p1p1}.
On the other hand, again by Example~\ref{example:p2-p1p1} the corresponding exceptional collection in $\bD(\FF_1)$ can be transformed by mutations to a collection of objects of ranks $(1,1,1,0)$,
therefore the original exceptional basis is not norm-minimal, and this case does not fit into our assumptions.

\begin{corollary}\label{corollary:minimal-4}
If a geometric surface-like pseudolattice $\GG$ with $n = 4$  has a norm-minimal exceptional basis with all ranks being non-zero, 
then $\GG$ is isometric to~$\NGr(\bD(\PP^1 \times \PP^1))$.
A norm-minimal exceptional basis in such $\GG$ corresponds to one of the collections 
\begin{equation*}
(\cO, \cO(1,0), \cO(c,1), \cO(c+1,1))
\end{equation*}
in $\bD(\PP^1 \times \PP^1)$ for some $c \in \ZZ$.
\end{corollary}

A combination of Corollary~\ref{corollary:minimal-n}, Corollary~\ref{corollary:minimal-3} and Corollary~\ref{corollary:minimal-4} proves Theorem~\ref{theorem:minimal-cats}.

\section{Minimal model program}\label{section:mmp}

Assume $\GG$ is a surface-like pseudolattice.

\subsection{Contraction}

Let $\ce \in \GG$ be an exceptional element of zero rank.
We consider the right orthogonal 
\begin{equation*}
\GG_\ce := \ce^\perp = \{ v \in \GG \mid \chi(\ce,v) = 0 \} \subset \GG.
\end{equation*}
Since $\chi(\ce,\ce) = 1$, we have a direct sum decomposition
\begin{equation}\label{eq:g-split}
\GG = \GG_\ce \oplus \ZZ\ce.
\end{equation}
Note that $\br(\ce) = 0$ means $\chi(\ce,\bp) = 0$, hence $\bp \in \GG_\ce$ and also~$\ce \in \bp^\perp$.
Abusing notation, we denote the projection of $\ce$ to $\NS(\GG) = \bp^\perp/\bp$ also by $\ce$.

\begin{lemma}\label{lemma:contraction}
The pseudolattice $\GG_\ce$ is surface-like and the element $\bp \in \GG_\ce$ is point-like.
Moreover, the rank function on $\GG_\ce$ is the restriction of the rank function on $\GG$;
we have an orthogonal direct sum decomposition
\begin{equation}\label{eq:contraction-ns}
\NS(\GG) =  \NS(\GG_\ce) \mathbin{\mathop{\oplus}\limits^\perp} \ZZ\ce,
\end{equation}
and a relation between the canonical classes
\begin{equation}\label{eq:contraction-k}
K_\GG = K_{\GG_\ce} + (-K_\GG \cdot \ce)\ce,
\end{equation}
If $\GG$ is unimodular or geometric, then so is $\GG_\ce$.
\end{lemma}
\begin{proof}
Clearly, $\bp$ is primitive in $\GG_\ce$, $\chi(\bp,\bp) = 0$, and $\chi_-(\bp,-)$ is zero on $\GG_\ce$.
Moreover, it is clear that the orthogonal of $\bp$ in $\GG_\ce$ is the intersection $\bp^\perp \cap \GG_\ce \subset \GG$, hence the form $\chi_-$ vanishes on it. 
This means that $\GG_\ce$ is surface-like with point-like element $\bp$.

Further, the direct sum decomposition~\eqref{eq:g-split} gives by restriction a direct sum 
\begin{equation*}
\bp^\perp = (\bp^\perp \cap \GG_\ce) \oplus \ZZ\ce,
\end{equation*}
and then by taking the quotient with respect to $\ZZ\bp$ a direct sum~\eqref{eq:contraction-ns}.
Its summands are mutually orthogonal by definition.

Furthermore, equality~\eqref{eq:chi-minus-k} shows that the orthogonal projection of~$K_\GG$ to $\NS(\GG_\ce)$ is the canonical class for $\GG_\ce$.
It fits into~\eqref{eq:contraction-k} since by Lemma~\ref{lemma:ne-zero-rank} we have $\ce^2 = -1$ as~$\ce$ is exceptional of zero rank.
Finally, unimodularity of $\GG_\ce$ is clear, and geometricity follows from~\eqref{eq:contraction-ns} and~\eqref{eq:contraction-k}.
\end{proof}

\begin{lemma}\label{lemma:mmp}
If $\GG$ is a surface-like pseudolattice then there is an exceptional sequence $\ce_1,\dots,\ce_k$ of rank zero elements 
such that the iterated contraction $\GG_{\ce_1,\dots,\ce_k}$ is a minimal surface-like pseudolattice.
\end{lemma}
\begin{proof}
Evidently follows by induction on the rank of $\GG$.
\end{proof}

The pseudolattice obtained from $\GG$ by an iterated contraction is called {\sf a minimal model} of $\GG$.
As we have seen, minimal geometric surface-like pseudolattices admitting an exceptional basis are isometric to numerical Grothendieck groups of $\PP^2$ or $\PP^1\times \PP^1$, 
so Lemma~\ref{lemma:mmp} can be thought of as a categorical minimal model program for surface-like pseudolattices of that kind.

\subsection{Defect}

In this section we define the defect of a lattice and of a surface-like pseudolattice.

\begin{definition}\label{def:defect}
Let $\rL$ be a lattice with a vector $K$.
The {\sf defect} of $(\rL,K)$ is defined as the integer
\begin{equation*}
\delta(\rL,K) := K^2 + \rk \rL - 10
\end{equation*}
If $\GG$ is a unimodular surface-like pseudolattice, we define its defect as $\delta(\GG) := \delta(\NS(\GG),K_\GG)$.
\end{definition}

It is easy to see that the defect is zero for numerical Grothendieck groups of surfaces with zero irregularity and geometric genus.

\begin{lemma}
Let $X$ be a smooth projective surface over an algebraically closed field of zero characteristic with $q(X) = p_g(X) = 0$.
The corresponding pseudolattice $\GG = \NGr(\bD(X))$ has zero defect $\delta(\GG) = 0$.
\end{lemma}
\begin{proof}
Since $H^1(X,\cO_X) = H^2(X,\cO_X) = 0$, we have $\rk \NS(X) = e - 2$, where $e$ is the topological Euler characteristic of $X$, and the holomorphic Euler characteristic $\chi(\cO_X)$ equals 1.
By Noether formula we have $e = 12\chi(\cO_X) - K_X^2$, hence $\rk \NS(\GG) = 10 - K_\GG^2$, hence $\delta(\GG) = 0$.
\end{proof}

In general, the defect of a surface-like pseudolattice can be both positive and negative.

\begin{example}
Let $X \to C$ be a ruled surface over a curve $C$ of genus $g$ 
(still under assumption that the base field is algebraically closed of zero characteristic)
and let $\GG = \NGr(\bD(X))$.
Then $\rk \NS(X) = 2$, while $K_X^2 = 8(1-g)$, hence $\delta(\GG) = 8(1-g) + 2 - 10 = -8g$.

Assume further that $X$ has a section $i \colon C \to X$ with normal bundle of degree $-1$ (for instance, $X$ can be the projectivization of a direct sum $\cO_C \oplus \cL$, where $\deg \cL = -1$). 
Set $\ce$ to be the class of the sheaf~$i_*\cO_C$ in the numerical Grothendieck group $\NGr(\bD(X))$.
Clearly, it is exceptional of rank~0, and $K_X \cdot \ce = 2g - 1$.
Consequently, by Lemma~\ref{lemma:contraction-defect} below the contraction $\GG_\ce$ has defect
\begin{equation*}
\delta(\GG_\ce) = \delta(\GG) - (1 - (K_X \cdot \ce)^2) = -8g - (1 - (2g - 1)^2) = 4g(g - 3).
\end{equation*}
In particular, it is negative for $g = 1$ and $g = 2$, zero for $g = 0$ and $g = 3$, and positive for $g > 3$.
\end{example}

As a result of classification of Theorem~\ref{theorem:minimal-cats} (more precisely, see Lemma~\ref{lemma:n=3} and Lemma~\ref{lemma:n=4}) we have

\begin{corollary}
If $\GG$ is a minimal geometric surface-like pseudolattice admitting an exceptional basis, then $\delta(\GG) = 0$.
\end{corollary}

An important property is that the defect does not decrease under contractions of geometric categories.

\begin{lemma}\label{lemma:contraction-defect}
Assume $\GG$ is a surface-like pseudolattice and $\ce \in \GG$ is exceptional of zero rank.
Then 
\begin{equation}\label{eq:contraction-defect}
\delta(\GG) = \delta(\GG_\ce) + (1 - (K_\GG \cdot \ce)^2).
\end{equation}
In particular, if $\GG$ is geometric, then $\delta(\GG) \le \delta(\GG_\ce)$ and this becomes an equality if and only if $K_\GG \cdot \ce = \pm 1$.
\end{lemma}
\begin{proof}
Equality~\eqref{eq:contraction-defect} easily follows from Lemma~\ref{lemma:contraction}.
Further, note that $K_\GG \cdot \ce \equiv \ce^2 \pmod 2$ since $K_\GG$ is characteristic, hence $K_\GG \cdot \ce$ is an odd integer, 
hence the second summand in the right hand side of~\eqref{eq:contraction-defect} is non-positive.
This proves the required inequality of the defects.
\end{proof}

\subsection{Exceptional bases in geometric surface-like pseudolattices}

Combining the minimal model program with the classification result of Theorem~\ref{theorem:minimal-cats} we get the following results.

\begin{theorem}[cf.~\protect{\cite[Corollary~9.12 and Corollary~10.7]{Pe}}]
\label{theorem-ranks-0-1}
Let $\GG$ be a geometric surface-like pseudolattice.
Any exceptional basis in $\GG$ can be transformed by mutations into an exceptional basis consisting of~$3$ or $4$ elements of rank~$1$ and all other elements of rank $0$.
\end{theorem}
\begin{proof}
First, we mutate the exceptional basis to a norm-minimal basis $\ce_1,\dots,\ce_n$.
Next, we apply a sequence of right mutations to this collection according to the following rule:
if a pair $(\ce_i,\ce_{i+1})$ is such that $\br(\ce_i) \ne 0$ and $\br(\ce_{i+1}) = 0$, we mutate it to $(\ce_{i+1},\RR_{\ce_{i+1}}\ce_i)$.
Then $\br(\RR_{\ce_{i+1}}\ce_i) = \br(\ce_i)$ by~\eqref{eq:mutations}, so the new collection is still norm-minimal.
It is also clear that after a number of such operations we will have the following property:
there is $1 \le k \le n$ such that 
\begin{equation}\label{eq:ranks-0-1}
\br(\ce_1) = \dots = \br(\ce_k) = 0,
\qquad \text{and}\qquad 
\br(\ce_i) \ne 0\quad\text{for $i > k$}.
\end{equation}
We consider the iterated contraction $\GG' = \GG_{\ce_1,\dots,\ce_k}$ of $\GG$. 
Clearly, $\ce_{k+1},\dots,\ce_n$ is an exceptional basis in $\GG'$.
Furthermore, it is norm-minimal. 
Indeed, if there is a mutation in $\GG'$ decreasing the norm of the collection, then the same mutation in $\GG$ would also decrease the norm in the same way.
Since the ranks of all elements in the collection of $\GG'$ are non-zero, and $\GG'$ is geometric by Lemma~\ref{lemma:contraction}, we conclude by Theorem~\ref{theorem:minimal-cats} 
that the ranks of elements $\ce_{k+1}, \dots, \ce_n$ are $\pm1$ and $n - k$ is equal to $3$ or $4$.
Changing the signs of elements with negative ranks, we deduce that the obtained collection in $\GG$ consists of elements of rank 0 and 1 only.
\end{proof}

\begin{theorem}[cf.~\protect{\cite[Theorem~10.8]{Pe}}]
\label{theorem-ranks-1}
Let $\GG$ be a geometric surface-like pseudolattice with \hbox{$\delta(\GG) = 0$}.
Any exceptional basis in~$\GG$ can be transformed by mutations into an exceptional basis consisting of elements of rank $1$.
\end{theorem}
\begin{proof}
By Theorem~\ref{theorem-ranks-0-1} there is an exceptional basis in $\GG$ satisfying~\eqref{eq:ranks-0-1} with $\br(\ce_i) = 1$ for $i > k$.
Note also that the category $\GG' = \GG_{\ce_1,\dots,\ce_k}$ has zero defect by Theorem~\ref{theorem:minimal-cats}.
Hence, by Lemma~\ref{lemma:contraction-defect} we have $K_\GG \cdot \ce_i = \pm 1$ for all $1 \le i \le k$.

We apply to this basis a sequence of right mutations according to the following rule:
if a pair $(\ce_i,\ce_{i+1})$ is such that $\br(\ce_i) = 0$ and $\br(\ce_{i+1}) = 1$, we mutate it to $(\ce_{i+1},\RR_{\ce_{i+1}}\ce_i)$.
Let us show that $\br(\RR_{\ce_{i+1}}\ce_i) = \pm1$. 
Indeed, by~\eqref{eq:chi-minus-k} we have
\begin{equation*}
\chi(\ce_i,\ce_{i+1}) = \chi_-(\ce_i,\ce_{i+1}) = - K_\GG \cdot \bl(\ce_i,\ce_{i+1}).
\end{equation*}
By rank assumptions and~\eqref{eq:def-lambda} we have $\bl(\ce_i,\ce_{i+1}) = -\ce_i$, hence $\chi(\ce_i,\ce_{i+1}) = K_\GG \cdot \ce_i = \pm 1$.
Therefore by~\eqref{eq:mutations} we have $\br(\RR_{\ce_{i+1}}\ce_i) = \chi(\ce_i,\ce_{i+1})\br(\ce_{i+1}) - \br(\ce_i) = \pm1$. 
If the rank is $-1$, we change the sign of the element to make the rank equal~1.

Clearly, after a finite number of such mutations we will get an exceptional basis consisting of rank~1 elements only.
\end{proof}

\subsection{Criterion for existence of an exceptional basis}

It is easy to show that the criterion for existence of a numerical exceptional collection in the derived category of a surface proved by Charles Vial in~\cite{V} also works for surface-like pseudolattices.
We start with a simple lemma.

\begin{lemma}\label{lemma:chi-1}
Let $(\GG,\chi)$ be a pseudolattice with $K_\GG$ integral and characteristic.
If $\br(v_1) = \br(v_2) = 1$ then $\chi(v_1,v_1) \equiv \chi(v_2,v_2) \pmod 2$.
Moreover, if $\chi(v,v)$ is odd for some $v \in \GG$ with $\br(v) = 1$, then for an appropriate choice of~$v$ one has $\chi(v,v) = 1$.
\end{lemma}
\begin{proof}
We have $v_2 = v_1 + D$, where $\br(D) = 0$. Therefore
\begin{equation*}
\chi(v_2,v_2) = \chi(v_1,v_1) + \chi(v_1,D) + \chi(D,v_1) + \chi(D,D).
\end{equation*}
Furthermore, by~\eqref{eq:chi-minus-k} we have
\begin{equation*}
\chi(v_1,D) + \chi(D,v_1) \equiv K_\GG \cdot \bl(v_1,D) = K_\GG \cdot D \pmod 2,
\end{equation*}
and by Lemma~\ref{lemma:chi-ker-r} we have $\chi(D,D) \equiv D^2 \pmod 2$.
Therefore
\begin{equation*}
\chi(v_2,v_2) - \chi(v_1,v_1) \equiv K_\GG \cdot D + D^2 \equiv 0 \pmod 2,
\end{equation*}
since $K_\GG$ is characteristic.

For the second part note that
\begin{equation*}
\chi(v+t\bp,v+t\bp) = \chi(v,v) + 2t,
\end{equation*}
so if $\chi(v,v)$ is odd, an appropriate choice of $t$ ensures that $\chi(v+t\bp,v+t\bp) = 1$.
\end{proof}

The condition that $\chi(v,v)$ is odd for a rank 1 vector $v \in \GG$ is thus independent of the choice of $v$ and is equivalent to existence of a rank 1 vector $v$ with $\chi(v,v) = 1$.
If it holds we will say that {\sf $(\GG,\chi)$ represents~$1$ by a rank~$1$ vector}.
This is a pseudolattice analogue of Vial's condition $\chi(\cO_X) = 1$ for a surface $X$.
In terms of the matrix representation~\eqref{eq:chi-matrix} of a pseudolattice this condition can be rephrased as $a \equiv 1 \pmod 2$.

\begin{theorem}[\protect{cf.~\cite[Theorem~3.1]{V}}]\label{theorem:criterion}
Let $(\GG,\chi)$ be a unimodular geometric pseudolattice of rank $n \ge 3$ and zero defect such that $(\GG,\chi)$ represents $1$ by a rank~$1$ vector.
Then we have the following equivalences:
\begin{enumerate}
\item $n = 3$ and $K_\GG = -3H$ for some $H \in \NS(\GG)$ if and only if $\GG$ is isometric to $\NGr(\bD(\PP^2))$;
\item $n = 4$, $\NS(\GG)$ is even and $K_\GG = -2H$ for some $H \in \NS(\GG)$ if and only if $\GG$ is isometric to $\NGr(\bD(\PP^1 \times \PP^1))$;
\item $n \ge 4$, $\NS(\GG)$ is odd and $K_\GG$ is primitive if and only if $\GG$ is isometric to $\NGr(\bD(X_{n-3}))$, where $X_{n-3}$ is the blowup of $\PP^2$ in $n-3$ points.
\end{enumerate}
Furthermore, $\GG$ has an exceptional basis if and only if one of the three possibilities listed above is satisfied.
\end{theorem}
\begin{proof}
The fact that for the surfaces $\PP^2$, $\PP^1 \times \PP^1$ and $X_{n-3}$ (with the standard surface-like structure) the numerical Grothendieck group has the properties listed in (1), (2), and (3) is evident.
Let us check the converse.

Assume $n = 3$ and $K_\GG = -3H$. 
By zero defect assumption we have $9H^2 = K_\GG^2 = 12 - 3 = 9$, hence, using the signature assumption we have $H^2 = 1$ and $H$ is primitive.
By Lemma~\ref{lemma:chi-1} and Remark~\ref{remark:matrix} with appropriate choice of a vector $v_0 \in \GG$ the matrix of $\chi$ in the basis $(v_0,H,\bp)$ has form
\begin{equation*}
\chi = \left(
\begin{smallmatrix}
1 & 3 & 1 \\ 0 & -1 & 0 \\ 1 & 0 & 0
\end{smallmatrix}
\right)
\end{equation*}
The above matrix is equal to the matrix of $\chi$ in $\NGr(\bD(\PP^2))$ in the basis $(\cO(-2H),\cO_L,\cO_P)$, where $L$ is a line and $P$ is a point, hence $\GG$ is isometric to $\NGr(\bD(\PP^2))$.

Assume $n = 4$, $\NS(\GG)$ even and $K_\GG = -2H$. 
By zero defect we have $4H^2 = K_\GG^2 = 12 - 4 = 8$, hence $H^2 = 2$ and $H$ is primitive.
By the parity and signature assumption $\NS(\GG)$ is the hyperbolic lattice, hence we may write $H = H_1 + H_2$, where $(H_1,H_2)$ is the standard hyperbolic basis.
By Lemma~\ref{lemma:chi-1} and Remark~\ref{remark:matrix} with appropriate choice of a vector $v_0 \in \GG$ the matrix of $\chi$ in the basis $(v_0,H_1,H_2,,\bp)$ has form
\begin{equation*}
\chi = \left(
\begin{smallmatrix}
1 & 2 & 2 & 1 \\ 0 & 0 & -1 & 0 \\ 0 & -1 & 0 & 0 \\ 1 & 0 & 0 & 0 
\end{smallmatrix}
\right)
\end{equation*}
The above matrix is equal to the matrix of $\chi$ in $\NGr(\bD(\PP^1 \times \PP^1))$ in the basis $(\cO(-H_1 - H_2),\cO_{L_1}, \cO_{L_2}, \cO_P)$, where $L_1$ and~$L_2$ are the two rulings and $P$ is a point, 
hence $\GG$ is isometric to $\NGr(\bD(\PP^1 \times \PP^1))$.

Finally, assume $n \ge 4$, $\NS(\GG)$ is odd and $K_\GG$ is primitive.
By~\cite[Proposition~A.12]{V} there is a basis $\ce_1,\dots,\ce_{n-3},H$ in $\NS(\GG)$ such that 
\begin{equation*}
\ce_i \cdot \ce_j = - \delta_{ij},
\qquad
H^2 = 1,
\qquad
\ce_i \cdot H = 0,
\qquad\text{and}\qquad
K_\GG = -3H + \sum \ce_i.
\end{equation*}
By Lemma~\ref{lemma:chi-1} and Remark~\ref{remark:matrix} with appropriate choice of a vector $v_0 \in \GG$ the matrix of $\chi$ in the basis $(v_0,H,\ce_1,\dots,\ce_{n-3},\bp)$ has form
\begin{equation*}
\chi = \left(
\begin{smallmatrix}
1 & 3 & 1 & \dots & 1 & 1 \\ 
0 & -1 & 0 & \dots & 0 & 0 \\ 
0 & 0 & 1 & \dots & 0 & 0 \\ 
\vdots & \vdots & \vdots & \ddots & \vdots & \vdots \\ 
0 & 0 & 0 & \dots & 1 & 0 \\ 
1 & 0 & 0 & \dots & 0 & 0 
\end{smallmatrix}
\right)
\end{equation*}
This matrix is equal to the matrix of $\chi$ in $\NGr(\bD(X_{n-3}))$ in the basis $(\cO(-2H),\cO_{L}, \cO_{E_1}, \dots, \cO_{E_{n-3}}, \cO_P)$, 
where $L$ is a general line on $\PP^2$ and $E_1$, \dots, $E_{n-3}$ are the exceptional divisors, hence $\GG$ is isometric to~$\NGr(\bD(X_{n-3}))$.

Now let us prove the second part of the theorem.

Assume $\GG$ has an exceptional basis. 
If $\GG$ is minimal, then by Theorem~\ref{theorem:minimal-cats} we know that $\GG$ is isometric to $\bD(\PP^2)$ or $\bD(\PP^1 \times \PP^1)$, hence either (1) or (2) holds.
If $\GG$ is not minimal, by Theorem~\ref{theorem-ranks-0-1} the exceptional basis can be transformed by mutations into 
a norm-minimal exceptional collection such that $\br(\ce_i) = 0$ for~$i \le k$ and $\br(\ce_i) = 1$ for $i \ge k + 1$ and $n - k \le 4$, $k \ge 1$.
We have $\chi(\ce_n,\ce_n) = 1$, hence $\chi$ represents~$1$ by a rank 1 vector.
Moreover, $\ce_1^2 = -1$, hence $\NS(\GG)$ is odd, and as it was shown in the proof of Theorem~\ref{theorem-ranks-1}, we have $K_\GG \cdot \ce_1 = \pm 1$, hence $K_\GG$ is primitive.
Thus (3) holds.
\end{proof}

\end{document}